\documentclass[final, disable, 12pt, a4paper]{article}
\pdfoutput=1

\usepackage{\jobname--additional_macros}
\usepackage{natbib}
\usepackage{lscape}
\usepackage{rotating}

\DeclareMathOperator{\lc}{lc}

\makeatletter
\newlength{\Stmtindent}
\newenvironment{marks-algorithm}[1] {
\setlength{\Stmtindent}{20pt}
\def\Stmt{\@ifnextchar[{\@Stmt}{\@Stmt[]}}
\def\@Stmt[##1]{\par\hspace*{\Stmtindent}\llap{##1\enspace\hfil}%
\hangindent\Stmtindent\ignorespaces}
\def\Inspec{\par\noindent\hangindent0pt\noindent%
\rlap{\@alginheader}\gdef\@alginheader{}%
\hspace*{2.75\Stmtindent}\llap{\raise2.5pt\hbox{\tiny$\blacktriangleright$}\enspace}%
\hangindent2.75\Stmtindent\ignorespaces}%
\def\Outspec{\par\hangindent0pt\noindent%
\rlap{\@algoutheader}\gdef\@algoutheader{}%
\hspace*{2.75\Stmtindent}\llap{\raise2.5pt\hbox{\tiny$\blacktriangleright$}\enspace}%
\hangindent2.75\Stmtindent\ignorespaces}
\def\>{\advance\hangindent\Stmtindent\hspace*{\Stmtindent}\ignorespaces}
\def\@alginheader{Input:} \def\@algoutheader{Output:}
\parindent=0pt \parskip=1pt
\medbreak\par {\bf Algorithm: }{\tt #1} } {\smallbreak}

\makeatother

\DeclareMathAlphabet{\mathbold}{OML}{cmm}{b}{it}

\newcommand{\avec}{{\vec a}}

\newcommand{\wvec}{{\vec{w}}}
\newcommand{\uvec}{{\vec{u}}}

\newcommand{\nxn}{{n\times n}}

\newcommand{\Fr}{\mathbb{F}_r}
\newcommand{\Fq}{\mathbb{F}_q}
\newcommand{\Frbar}{\overline{\mathbb{F}}_{r}}
\newcommand{\Fqbar}{\overline{\mathbb{F}}_{q}}
\newcommand{\Frbarn}{\overline{\mathbb{F}}^n_{r}}

\newcommand{\divs}{{\mskip3mu|\mskip3mu}}
\newcommand{\fhat}{{f^*}}

\def\iso{\cong}
\newcommand{\ghat}{{g^*}}
\newcommand{\hhat}{{h^*}}

\newcommand{\E}{{\mathsf{E}}}
\newcommand{\F}{{\mathsf{F}}}

\newcommand{\Aq}{{{\Fq[x; r]}}}
\newcommand{\cAq}{{\Fr[x;q]}}

\newcommand{\calB}{{\mathcal B}}

\newcommand{\Hide}[1]{}

\DeclareMathOperator{\diag}{diag}
\DeclareMathOperator{\gcrc}{gcrc}
\DeclareMathOperator{\lclc}{lclc}
\DeclareMathOperator{\mclc}{mclc}

\begin{document}




\title{Composition collisions \\
and projective polynomials 
}
\author{Joachim von~zur~Gathen \\
B-IT, Universit\"at Bonn\\
D-53113 Bonn, Germany\\
\email{gathen@bit.uni-bonn.de}\\
\url{http://cosec.bit.uni-bonn.de/} \\
\and
Mark Giesbrecht\\
Cheriton School of Computer Science\\
University of Waterloo, Waterloo, ON, N2L 3G1 Canada\\
\email{mwg@cs.uwaterloo.ca}\\
\url{http://www.cs.uwaterloo.ca/~mwg}\\
\and
Konstantin Ziegler\\
B-IT, Universit\"at Bonn\\
D-53113 Bonn, Germany\\
\email{zieglerk@bit.uni-bonn.de}\\
\url{http://cosec.bit.uni-bonn.de/}\\
}

\maketitle

\begin{abstract}
The functional decomposition of polynomials has been a topic of great
interest and importance in pure and computer algebra and their applications.
The structure of compositions of (suitably normalized) polynomials $f = g
\circ h$ in
$\FF_q[x]$ is well understood in many cases, but quite poorly when the degrees of both
components
are divisible by the characteristic $p$.  This work investigates the
decomposition of polynomials whose degree is a power of $p$. 
An (equal-degree) \emph{$i$-collision} is a set of $i$ distinct pairs $(g,h)$ of
polynomials, all with the same composition and $\deg g$ the same for all
$(g,h)$.  \cite{abh97} introduced the \emph{projective polynomials}
$x^{n}+ax+b$, where $n$ is of the form $(r^{m}-1)/(r-1)$.  Our first tool is a bijective correspondence between
$i$-collisions of certain additive trinomials, projective polynomials with $i$
roots, and linear spaces with $i$ Frobenius-invariant lines.

\cite{blu04a} has determined the possible number of roots of projective
polynomials for $m=2$, and how many polynomials there are with a prescribed
number of roots.  We generalize her first result to arbitrary $m$, and provide
an alternative proof of her second result via elementary linear algebra.

If one of our additive trinomials is given, we can efficiently compute the
number of its decompositions, and similarly the number of roots of a
projective polynomial.  The runtime of these algorithms depends polynomially on the sparse
input size, and thus on the input degree only logarithmically.

For non-additive polynomials, we present certain decompositions and conjecture
that these comprise all of the prescribed shape.
\end{abstract}

\textbf{Keywords.}   Univariate polynomial decomposition, additive
polynomials, projective polynomials.

\textbf{2010 Mathematics Subject Classification.}  Primary 68W30;
Secondary 12Y05


\section{Introduction}

The \emph{composition} of two polynomials $g,h \in F[x]$ over a field
$F$ is denoted by $f= g \circ h= g(h)$, and then $(g,h)$ is a
\emph{decomposition} of $f$. In the 1920s, Ritt, Fatou, and Julia
studied structural properties of these decompositions over
$\mathbb{C}$, using analytic methods. Particularly important are two
theorems by Ritt on uniqueness, in a suitable sense, of
decompositions, the first one for (many) indecomposable components and
the second one for two components, as above.

The theory was algebraicized by \cite{dorwha74}, \cite{sch82c,
sch00c}, and others. Its use in a cryptographic context was
suggested by \cite{cad85}. In computer algebra, the method of
\cite{barzip85} requires exponential time but works in all
situations.  A breakthrough result of \cite{kozlan89} was their
polynomial-time algorithm to compute decompositions. One has to distinguish  between the \emph{tame case}, where the characteristic $p$
does not divide $\deg g$ and this algorithm works (see \cite{gat90c}),
and the \emph{wild case}, where $p$ divides $\deg g$ (see
\cite{gat90d}).  In the wild case, considerably less is known, mathematically and computationally. The algorithm of \cite{zip91} for
decomposing rational functions suggests that the block decompositions
of \cite{lanmil85} (for determining subfields of algebraic number
fields) can be applied to the wild case.  \cite{Gie98} provides
fast algorithms for the decomposition of additive (or linearized)
polynomials, in some sense an ``extremely wild'' case.  We exploit
 their elegant structure here.  An
enumeration of number or structure of solutions in the wild case has
defied both algebraic and computational analysis, and we attempt to
address this here.  Moreover, many of the algorithms we present here
are sensitive to the sparse size of the input, as opposed to the
degree, a property not exploited in the above-mentioned papers.

The task of counting compositions over a finite field of characteristic $p$
was first considered in \cite{gie88}. \Citet{gat09b} presents general
approximations to the number of decomposable polynomials. These come
with satisfactory (rapidly decreasing) relative error bounds except
when $p$ divides $n = \deg f$ exactly twice. The goal of the present
work is to study the easiest of these
difficult cases, namely when $n = p^{2}$ and hence $\deg g =\deg h =
p$.  However, many of our results are valid for $n=r^{2}$ for a power $r$ of
$p$, and are stated accordingly.

We introduce the notion of an equal-degree \emph{$i$-collision} of
decompositions, which is a set of $i$ pairs $(g,h)$, all with the same
composition and $\deg g$ the same for all $(g,h)$.  These are the only
collisions we consider in this paper, and we omit the adjective
``equal-degree'' in the text.  An $i$-collision is \emph{maximal} if it is not
contained in an $(i+1)$-collision.  After some preliminaries in
\autoref{sec:InNo}, we start in \autoref{sec:proj-polyn} with the particular
case of additive polynomials. We relate the decomposition question to one
about eigenspaces of the linear function given by the Frobenius map on the
roots of $f$. This yields a complete description of all decompositions of
certain additive trinomials in terms of the roots of the \emph{projective
  polynomials} $x^{n}+ax+b$, introduced by \cite{abh97}, where $n$ is of the
form $(r^{m}-1)/(r-1)$.  We prove that maximal $i$-collisions of additive
polynomials of degree $r^{2}$ exist only when $i$ is $0$, $1$, $2$ or $r+1$,
count their numbers exactly, and show their relation to the roots of
projective polynomials for $m=2$.  In this case \cite{blu04a} has determined,
the number of roots that can occur, namely $0$, $1$, $2$, or $r+1$, and also
for how many coefficients $(a,b)$ each case happens. We obtain elementary
proofs of a generalization of her first result to arbitrary $m$ and of her
counts for $m=2$.  From the proof we obtain a fast algorithm (polynomial in
$r$ and $\log q$) to count the number of roots over $\FF_{q}$, called
\emph{rational} roots.  More generally, in \autoref{sec:algos} an algorithm is
provided to enumerate the possible number of right components of an additive
polynomial of any degree.  A fast algorithm is then presented to count the
number of right components of an additive polynomial of any degree, which is
shown to be equivalent to counting rational roots of projective polynomials of
arbitrary degree.  We also demonstrate theorems and fast algorithms to count
and construct indecomposable additive polynomials of prescribed degree.  In
\autoref{sec:proj-polyn-roots} we actually construct and enumerate all
additive polynomials of degree $r^{2}$ with 0, 1, 2, or $r+1$ collisions and establish connections to the
counts of \cite{blu04a} and \cite{gat09d}.

In \autoref{sec:general} we move from additive to general polynomials.
Certain $(r+1)$-collisions are derived from appropriate roots of projective
polynomials. We conjecture that these are all possibilities and present
results on general $i$-collisions with $i \geq 2$ for $r=p$ that support our
conjecture.


\section{The basic setup}
\label{sec:InNo}

We consider polynomials $f,g,h\in\FF_q[x]$ over a finite field $\FF_q$ of
characteristic $p$.  Then $f= g \circ h = g(h)$ is the \emph{composition} of
$g$ and $h$, $(g,h)$ is a \emph{decomposition} of $f$, and $g$ and
$h$ are a \emph{left} and \emph{right component}, respectively, of $f$.  Furthermore, $f$ is
\emph{decomposable} if such $(g,h)$ exist with $\deg g, \deg h \geq
2$, and \emph{indecomposable} otherwise.

We call $f$ \emph{original} if its graph passes through the origin,
that is, if $f(0)=0$.  Composition with linear polynomials introduces
inessential ambiguities in decompositions.  If $f = g\circ h, a \in
\FF_{q}^{\times}$,
and $b\in \FF_{q}$, then $af+b=(ag+b)\circ h$.  Thus we may assume $f$
to be monic original.  Furthermore, if $a = \lc(h)^{-1}$ and
$b=-ah(0)$, then $f=g\circ h = g((x-b)a^{-1})\circ (ah+b)$ and the right
component is monic original.  Thus we
may also assume $h$ to be monic original, and then $g$ is
so automatically.  We thus consider the following two sets:
\begin{align*}
P_n (\FF_{q}) & =  \{ f \in \FF_q[x] \colon \text{$f$ is monic and
original of degree $n$}\}, \\
D_n (\FF_{q}) & = \{ f \in P_n (\FF_{q}) \colon \text{$f$ is
decomposable}\}.
\end{align*}
We usually leave out the argument $\FF_{q}$.  The size of the first set is
$\# P_n = q^{n-1}$,
and determining (exactly or approximately) $\# D_{n}$ is one of the
goals in this business.  The number of all or all decomposable
polynomials of degree $n$, not restricted to $P_{n}$, is $\# P_{n}$ or $\#
D_{n}$, respectively, multiplied by $q(q-1)$.


First, we consider the additive or linearized polynomials, which have a mathematically rich and
highly useful structure in finite fields.  First introduced in
\cite{Ore33b}, they play an important role in the theory of finite and
function fields, and they have found many applications in codes and
cryptography.  See \cite{LidNie83}, Chapter 3, for an introduction and
survey over finite fields.

We will focus on additive polynomials over finite fields, though some
of these results will hold more generally in characteristic $p$.  For
convenience we assume that $r$ is a power of $p$ and $q=r^d$ for some
$d\in\ZZ_{>0}$.  Let
\[
\Aq = \{ \sum_{0\leq i\leq n} a_i x^{r^i} \colon
n\in\ZZ_{\geq 0},\ a_0,\ldots,a_n\in \Fq \}
\]
be the ring of \emph{$r$-additive} (or \emph{linearized}, or simply
\emph{additive}) polynomials over $\Fq$.  These are the polynomials
such that $f(\alpha a+\beta b)=\alpha f(a)+\beta f(b)$ for any
$\alpha,\beta\in\Fr$, and for any $a,b\in\Fqbar$, where $\Fqbar$ is an
algebraic closure of $\Fq$.  The additive polynomials form a
(non-commutative) ring under the usual addition and composition. It is
a principal left (and right) ideal ring with a left (and right)
Euclidean algorithm.

An additive polynomial is squarefree if $f'$ (the derivative of $f$)
is nonzero, meaning that the linear coefficient
of $f$ is nonzero.  If $f\in\Aq$ is squarefree of degree $r^n$, then
the set of all roots of $f$ form an $\Fr$-vector space in $\Frbar$ of
dimension $n$.  Conversely, for any finite dimensional $\Fr$-vector
space $W\subseteq\Frbar$, the lowest degree polynomial
$f = \prod_{a \in W} (x-a) \in\Frbar[x]$ with $W$ as its roots is a squarefree
$r$-additive polynomial.  Let $\sigma_{q}$ denote
the $q$th power Frobenius automorphism on $\Fqbar$ over $\Fq$.  If $W$ is invariant under $\sigma_{q}$, then $f\in\Aq$.

We have 
\begin{equation*}
x^{p} \circ h = \sigma_{p} (h) \circ x^{p}  
\end{equation*}
for $h \in \Fq[x]$, where
$\sigma_{p}$ is the Frobenius automorphism on $\FF_{q}$ over $\FF_{p}$, which
extends to polynomials coefficientwise.  If $\deg h = p$ and $h \neq x^{p}$, this is a
$2$-collision and called a \emph{Frobenius
collision}.  It is never part of $i$-collisions with $i\geq 3$.

\begin{lemma}
\label{lem:components_vs_subspaces}  Let $S \in \Fr^{n \times n}$ be the
matrix representing the Frobenius $\sigma_{q}$.  There is a bijection
between $S$-invariant subspaces of $\Fr^{n\times 1}$ and right
components $h\in\Aq$ of $f$.
\end{lemma}

\begin{proof}
Assume that $f\in\Aq$ is squarefree of degree $r^n$.  
Let $v_1,\ldots,v_n\in\Frbar$ form an $\Fr$-basis for $V_f$,
and identify $a=\sum_{1\leq i\leq n}\alpha_i v_i\in V_f$ with
$\avec=(\alpha_1,\ldots,\alpha_n)\in\Fr^n$.  Each $\Fr$-subspace $W$
of $V_f$ corresponds to an additive right component $h$ of
$f$ which has $W$ as its set of roots. It is relatively straightforward to
derive that all components of an additive polynomial are again
additive \cite[Theorem 3.3]{gie88}.  Finally, we have $h \in \Aq$ if and only if $W$ is
invariant under $\sigma_{q}$.


Generally, if $f \in \Aq$ is not squarefree, we can write it as $f=g\circ
x^{r^t}$ for a squarefree $g\in\Aq$, and then $f=x^{r^t}\circ h$ for
some squarefree $h\in\Aq$ (see \cite{gie88}, Sections 3--4).
\end{proof}

We
present two related approaches to investigate $f \in \Aq$ of degree $r^{2}$.  The first, working with normal forms
of the Frobenius operator on the space of roots of $f$, gives a
straightforward classification of the number of possible
decompositions, though provides less insight into how many polynomials
fall into each class.  The second uses more structural information
about the ring of additive polynomials and provides complete
information on both the number of decompositions and the number of
polynomials with each type of decomposition.

We can easily classify all possible collisions
in the non-squarefree case at degree $r^{2}$ as follows.

\begin{lemma}
\label{lem:squarefull}
Let $f=x^{r^2}+ax^r\in\Aq$ for $a\in\Fq$.  Then $f$ has a
$2$-collision if $a\neq 0$ and a unique decomposition if $a=0$.
\end{lemma}



Closely related to decompositions are the following objects.  Let $r$ be a power of $p$ and $m\geq 2$.  \cite{abh97} introduced the
\emph{projective polynomials}
\begin{equation*}
\Psi_m^{(a,b)}=x^{(r^{m}-1)/(r-1)}+ax+b  
\end{equation*}
which have, over
appropriate fields, nice Galois groups such as general linear or
projective general linear groups.  We assume $q$ to be a power of $r$, and
have for $m=2$
\begin{equation}
\label{eq:proj}
\Psi_2^{(a,b)}=x^{r+1}+ax+b
\end{equation}
with $a,b\in\FF_{q}$.

In the case $ab \neq 0$, \cite{blu04a} has proven an amazingly precise
result about the number of nonzero roots of \eqref{eq:proj}.  Namely,
this number is $0, 1, 2$, or $r+1$, and she has exactly determined the
number of parameters $(a,b)$ for which each of the four possibilities
occurs.  In the case $a = 0$, the corresponding number is given in
\cite{gat08c}, Lemma~5.9.



Projective polynomials appear naturally in many situations.  \cite{Blu04b} used them to construct strong
Davenport pairs explicitly and \cite{Dil02} to build families of
difference sets with certain Singer parameters.  \cite{Blu03} proved the
equivalence of two such difference sets, using again projective polynomials
and they played a central role in tackling the question of when a quartic power
series over $\Fq$ is actually hyperquadratic \citep{Blulas06}.

\cite{HelKho08} used projective polynomials to find $m$-sequences of length
$2^{2k}-1$ and $2^{k}-1$.  \cite{HelKho10} studied projective polynomials
further, providing criteria for the number of zeros in a field of
characteristic $2$, not assuming $q$ to be a power of $r$. \cite{ZenLi08} applied the techniques of \cite{blu04a} to study the roots of
$\delta^{p^{n-k}}y^{p^{n/2-k}+1}+\gamma y + \delta$ with $\delta\gamma \neq 0$
to define a class of $p$-ary codes $C$, where $p$ is an odd prime, and completely
determine their weight distribution.

\section{Additive and projective polynomials}
\label{sec:proj-polyn}



We assume that $q=r^d$ and $r$ is a power of the characteristic $p$ of $\Fq$.
In this section we establish a general connection between decompositions of
certain additive polynomials and roots of projective polynomials, and
characterize the possible numbers of rational roots of the latter.

\begin{lemma}
\label{lem:projadd}
Let $m\geq 1$, $f =x^{r^m} + ax^r + bx$ and $h=x^{r} - h_{0}x$ be in $\Aq$
with $a, b, h_{0} \in \Fq$.  Then
$f=g\circ h$ for some $g\in\Aq$ if and only if $\Psi_m^{(a,b)}(h_{0})=0$.
\end{lemma}

\begin{proof}  For $b=0$ the claim follows from \autoref{lem:squarefull}, and
  it is readily checked for $m=1$.  Now we assume $b\neq 0$, $m\geq
  2$, and consider $g_{0}, \dots, g_{m-2} \in \Fq$ satisfying
\begin{align*}
f & = x^{r^m} + ax^r + bx \\
& = \left(x^{r^{m-1}} + g_{m-2}x^{r^{m-2}} + \cdots +
g_1 x^r + g_0 x\right) \circ \left(x^r - h_{0}x\right).
\end{align*}
Equating coefficients yields
\begin{align*}
0 & = g_{m-2}-h_{0}^{r^{m-1}}, \\
0 & = g_{i-1} - g_{i} h_{0}^{r^{i}}, \quad \text{for $2 \leq i \leq m-2$}, \\
a & = g_{0} - g_{1} h_{0}^{r}, \\
b & = - g_{0} h_{0}.
\end{align*}
Thus $h_{0} \neq 0$ and 
\begin{align}
g_i = & h_{0}^{r^{i+1}+r^{i+2}+\cdots+r^{m-1}} \quad \text{for} \quad 1 \leq i
\leq m-2, \nonumber \\
g_0 = & h_{0}^{r+r^2+\cdots+r^{m-1}}+a  = -b /h_{0} \label{eq:14}.
\end{align}
Multiplying through by $h_{0}$ concludes the proof.
\end{proof}

This lemma and \autoref{lem:components_vs_subspaces} are the building blocks for the powerful
equivalences summarized as follows.

\begin{proposition}
\label{pro:equivalent}
  Let $r$ be a power of $p$, $m \geq 2$, $a,b \in \Fq$ and $f =
  x^{r^{m}}+ax^{r}+b$.  There is a one-to-one correspondence between any two of
  the following sets.
  \begin{itemize}
  \item right components of $f$ with degree $r$,
  \item roots of $\Psi_{m}^{(a,b)}$,
  \item $\sigma_{q}$-invariant linear subspaces of $V_{f}$ with dimension $1$.
  \end{itemize}
\end{proposition}

More generally, assume that $f\in\Aq$ is any additive polynomial of
degree $r^n$.  We now list the possible numbers of
right components in $\Aq$. A \emph{rational Jordan form} has the shape
\begin{equation}
\label{eq:jordanform}
\begin{aligned}
& S=\diag(J_{\alpha_1}^{e_{11}},\ldots,J_{\alpha_1}^{e_{1k_1}},\ldots,
          J_{\alpha_\ell}^{e_{\ell1}},\ldots,J_{\alpha_\ell}^{e_{\ell
              k_\ell}}    )\in\Fr^{m\times m},\\
&\mbox{where}~ J_{\alpha_i}^{e_{ij}}=
\left(
\begin{smallmatrix}
C_{\alpha_i} & I_{s_i} &  \hbox to 0pt{\kern-5pt\vbox to 0pt{\kern-4pt\Large 0}}&  \\
             & \ddots & \ddots &   \\[-3pt]
             &  & \ddots & \raise4pt\hbox{\footnotesize $I_{s_i}$} \\
             &            &  & C_{\alpha_i}
\end{smallmatrix}
\right)
\in\Fr^{e_{ij}s_i\times e_{ij}s_i},
\end{aligned}
\end{equation}
and $\alpha_1,\ldots,\alpha_\ell\in\Frbar$ are the distinct
non-conjugate roots 
of the characteristic polynomial of $S$
(i.e., eigenvalues), $C_{\alpha_{i}}\in\F_r^{s_i\times s_i}$ is the companion
matrix of $\alpha_i$ (assuming $[\Fr[\alpha_i]:\Fr]=s_i$) and
$I_{s_i}$ is the $s_i\times s_i$ identity matrix.  

Following the proof of \autoref{lem:components_vs_subspaces}, let $V_f$ be the $\Fr$-vector space of roots,
and $S\in\Fr^{m\times m}$ the matrix representation of the Frobenius
operations $\sigma_{q}$ on $\Frbar$. 

\begin{proposition}[see, e.g. \cite{gie95}]
 Every matrix in $\Fr^{m\times m}$ is similar to one in rational Jordan form, and the number and multiplicity of eigenvectors
is preserved by this transformation.  
\end{proposition}

Thus, we may assume $S$ to be of the form
described in (\ref{eq:jordanform}).  Since we are only
interested here in $\sigma_{q}$-invariant subspaces of dimension 1, we
ignore for now all $\alpha_i$ which are not in $\Fr$.  The number of $A$-invariant lines --- one dimensional subspaces
invariant under $A$ --- is described as follows.

\begin{theorem}
  If $A\in\Fr^\nxn$ has rational Jordan normal form as in
  (\ref{eq:jordanform}), then the number of $A$-invariant lines in
  $\Fr^{n \times 1}$ is
 \[
 \sum_{\stackrel{1\leq i\leq k}{\alpha_i\in\Fr}}\prod_{1\leq j\leq
   k_i} \frac{r^{k_{ij}}-1}{r-1}.
 \]
\end{theorem}
\begin{proof}
  For each eigenvalue $\alpha_i\in\Fr$ ($1\leq i\leq \ell$) of $A$,
  the rational Jordan block $J_{\alpha_i}^{e_{ij}}$ has an eigenspace
  of dimension one. The entire eigenspace of $A$ associated with
  $\alpha_i$ has dimension $k_i$, and hence contains
  $(r^{k_i}-1)/(r-1)$ lines.  Since no line is associated with two
  distinct eigenvalues, we simply add the number of lines associated with each
  eigenvalue in $\Fr$.
\end{proof}

For example, in $\Fr^{3\times 3}$ we can
list all matrix classes and the number of $1$-dimensional
invariant subspaces as follows:
\[
\begin{array}{cccc}
\left(
\begin{smallmatrix}
\hbox{$\alpha_1$} &  & \\
& \hbox{$\alpha_1$} & \\
&  & \hbox{$\alpha_1$}
\end{smallmatrix}
\right)
,& 
\left(
\begin{smallmatrix}
\hbox{$\alpha_1$} & 1 & \\
& \hbox{$\alpha_1$} & \\
&  & \hbox{$\alpha_1$}
\end{smallmatrix}
\right)
,&
\left(
\begin{smallmatrix}
\hbox{$\alpha_1$} & 1 & \\
& \hbox{$\alpha_1$} & 1\\
&  & \hbox{$\alpha_1$}
\end{smallmatrix}
\right)
\\[3pt]
r^2+r+1 & r+1 &  1 \\[6pt]
\left(
\begin{smallmatrix}
\hbox{$\alpha_1$} & 1 & \\
& \hbox{$\alpha_1$} & \\
&  & \hbox{$\alpha_2$}
\end{smallmatrix}
\right)
,&
\left(
\begin{smallmatrix}
\hbox{$\alpha_1$} &  & \\
& \hbox{$\alpha_2$} & \\
&  & \hbox{$\alpha_3$}
\end{smallmatrix}
\right)
,&
\left(
\begin{smallmatrix}
\framebox{\hbox to 10pt{\vbox to 6pt{}}}  & \\
& \hbox{$\alpha_1$}
\end{smallmatrix}
\right)
,&
\left(
\vbox to 13pt{}\framebox{\lower1pt\hbox to 20pt{\vbox to 9pt{}}}
\right),\\
2 & 3 & 1 & 0
\end{array}
\]
where the number of $1$-dimensional invariant subspaces is listed
beneath each matrix.  Empty boxes indicate companion blocks associated
with eigenvalues not in $\Fr$.

For a positive integer $m$, let $\Pi_{m}$ be the set of partitions $\pi =
(s_{1}, \dots, s_{k})$ with positive integers $s_{i}$ and $s_{1} + \dots +
s_{k} = m$,
$\varphi_{r,m}=(r^m-1)/(r-1)$, for any $\pi \in \Pi_{m}$, let
$\varphi_{r}(\pi)=\varphi_{r, s_{1}} + \varphi_{r, s_{2}} + \dots +
\varphi_{r, s_{k}}$, and $\varphi_{r}(\Pi_{m}) = \{
\varphi_{r}(\pi) \colon \pi \in \Pi_{m}\}$.

\begin{theorem}
\label{thm:S}
We consider the set
  \begin{equation*}
    S_{q,r,m} = \{ i \in \NN \colon \exists f \in \Aq, \deg f = r^{m}, 
   f~\text{is a maximal $i$-collision} \}.
  \end{equation*}
of maximal collision sizes for additive polynomials.  Then
\begin{align*}
  S_{0} & = \{0\}, \\
  S_{m} & = S_{m-1} \cup \varphi_{r}(\Pi_{m}).
\end{align*}
 \end{theorem}

As examples, we have 
\begin{align*}
  S_{0} & = \{ 0 \}, \\
  S_{1} & =  S_{0} \cup \{ \varphi_{r}(1) \} = \{0, 1\}, \\
  S_2 & = S_{1} \cup \{ \varphi_{r}(1,1), \varphi_{r}(2) \} = \{0, 1, 2, r+1  \},~\mbox{(consistent with \cite{blu04a})}\\
  S_3 & =  S_2 \cup \{\varphi_{r}(3), \varphi_r(2)+1, 3\}, \\
  S_4 & = S_3 \cup \{\varphi_r(4), \varphi_r(3)+1, 2\varphi_r(2), \varphi_r(2)+2, 4\},\\
  S_5 & = S_4 \cup \{\varphi_r(5), \varphi_r(4)+1,
  \varphi_r(3)+\varphi_r(2),\varphi_r(3)+2, 2\varphi_r(2)+1, \\
  & \quad \quad \quad \varphi_r(2)+3, 5 \} ,\\
  S_6 & = S_5 \cup \{ \varphi_r(6), \varphi_r(5)+1, \varphi_r(4)+\varphi_r(2),
  \varphi_r(4)+2, 2\varphi_r(3),\\
  & \quad \quad \quad  \varphi_r(3)+\varphi_r(2)+1, \varphi_r(3)+3, 3\varphi_r(2), 2\varphi_r(2)+2, \varphi_r(2)+5, 6 \}. \\
\end{align*}
The size of $S_m$ equals $\sum_{0\leq k\leq m} p(k)$, where $p(k)$ is the
number of additive partitions of $k$. This grows exponentially in $m$
\citep{harram18} but is still surprisingly small considering the generality of
the polynomials involved.

\begin{corollary} Let $r$ be a power of $p$, $m \geq 0$, $a,b \in \Fq$ and $f=  x^{r^{m}}+ax^{r}+bx$.
  \begin{enumerate}
  \item The possible number of roots of
    $\Psi_{n}^{(a,b)}$ is $S_{m}$.
  \item The possible number of $\sigma_{q}$-invariant linear
  subspaces of $V_{f}$ of dimension $1$ is $S_{m}$.
  \end{enumerate}
\end{corollary}


We investigate the general result of
\autoref{thm:S} in the case $m=2$ further.  This leads to an exact
determination, for each $i$, of how often $i$-collisions occur; see
\autoref{cor:bluher}. Assume that $f\in\Aq$ is squarefree, with root space $V_f$.  Again let
$\sigma_{q}$ be the Frobenius automorphism fixing $\Fq$, and $S\in\Fr^{2\times 2}$
its representation with respect to some fixed basis.  The
number of one-dimensional subspaces of $V_f$ invariant under $\sigma_{q}$
is equal to the number of nonzero vectors $w\in\Fr^{2\times 1}$ such that $Sw=\lambda w$ for
some $\lambda\in\Fr$, that is, the number of eigenvalues of $S$.  Each
such $w$ generates a one-dimensional $\sigma_{q}$-invariant subspace, and
each such subspace is generated by $r-1$ such $w$.  Thus, the number
of distinct $\sigma_{q}$-invariant subspaces of dimension one, and hence
the number of right components in $\Aq$ of degree $r$, is
equal to the number of eigenvectors of $S$ in $\Fr^2$, divided by
$r-1$.

We now classify $\sigma_{q}$ according to the possible matrix similarity
classes of $S$, as captured by its rational canonical form, and count
the number of eigenvectors and components in each case.  Note
that the number of eigenvectors of $S$ equals the number of
eigenvectors of $T$ when $S$ is a similar matrix to $T$ ($S\sim T$).
\begin{theorem}
\label{thm:countdecomp}
Let $f\in\Aq$ be squarefree of degree $r^2$.  Suppose the Frobenius
automorphism $\sigma_{q}$ is represented by $S\in\Fr^{2\times 2}$, and
$\Lambda\in\Fr[z]$ is the minimal polynomial of the matrix $S$.
Then one of the following holds:
\begin{description}
\item[Case 0:] $\displaystyle S\sim \begin{pmatrix}
0 & \delta\\
1 & \gamma
\end{pmatrix}$, and $\Lambda=z^2-\gamma z-\delta\in\Fr[z]$ is
irreducible, and $f$ is indecomposable.
\item[Case 1:] $\displaystyle S\sim \begin{pmatrix}
\gamma & 1\\
0 & \gamma
\end{pmatrix}\in\Fr^{\mskip1mu 2\times 2}$ with $\gamma\neq
0$, and $\Lambda=(z-\gamma)^2$, and $f$ has a unique right
component of degree~$r$.
\item[Case 2:] $\displaystyle S\sim \begin{pmatrix}
\gamma & 0\\
0 & \delta
\end{pmatrix}\in\Fr^{\mskip1mu 2\times 2}$
for $\gamma\neq\delta$ with $\gamma\delta\neq 0$, when
$\Lambda=(z-\gamma)(z-\delta)$, and $f$ has a 2-collision.
\item[Case $\mathbold{r\mskip2mu +}$1:]
$\displaystyle S = \begin{pmatrix}
\gamma & 0 \\
0 & \gamma
\end{pmatrix}\in\Fr^{\mskip1mu 2\times 2}$, for $\gamma\neq
0$, and $f$ has an $(r+1)$-collision.
\end{description}
\end{theorem}

\begin{proof}\quad
  \begin{description}
\item[Case 0:] $S$ represents multiplication by $z$ in the finite field
$\E=\Fr[z]/(\Lambda)$.   However, there is no $a\in\E^\times$ such that
$za=\lambda a$ for $\lambda\in\Fr^\times$, so there are no
eigenvectors, and hence no right components of degree $r$.
\item[Case 1:]
Nonzero vectors of the form $(\alpha,0)\in\Fr^2$ are
eigenvectors, and there are $r-1$ of these.  Thus $f$ has
$(r-1)/(r-1)=1$ right components in $\Aq$ of degree
$r$.
\item[Case 2:]
Nonzero vectors of the form $(\alpha,0)\in\Fr^2$ and
$(0,\beta)\in\Fr^2$ are eigenvectors, and there are $2(r-1)$
of these.  Thus $f$ has $2(r-1)/(r-1)=2$ right composition
components in $\Aq$ of degree $r$.
\item[Case $\mathbold{r\mskip2mu +}$1:]
Every nonzero element of $\Fr^2$ is an eigenvector, and hence
there are $r^2-1$ of them, so $f$ has $(r^2-1)/(r-1)=r+1$ right
components in $\Aq$ of degree $r$.
\end{description}
\end{proof}

\section{Algorithms for additive polynomials}

\label{sec:algos}


Given $f\in\Aq$ of degree $r^2$, using the techniques of
\autoref{sec:proj-polyn}, combined with basic algorithms from
\cite{Gie98}, we can quickly determine the number of collisions for
$f$. 

The centre of $\Aq$ will be a useful tool in understanding its
structure, and is easily shown to be equal to
\[
\cAq=\left\{\sum_{0\leq i\leq\kappa} a_ix^{q^{i}} \colon \kappa \in \NN, 
     a_0,\ldots,a_\kappa\in\Fr \right\}
\subseteq\Aq
\]
(see, e.g., \cite{Gie98}).  This is isomorphic to
the ring $\Fr[y]$ of polynomials under the usual addition and multiplication,
via the isomorphism
\[
f = \sum_{0\leq i\leq\kappa} a_ix^{q^{i}}
~\mapsto~
\tau(f) = \sum_{0\leq i\leq \kappa} a_iy^i
\]
(see \cite{LidNie83}, Section 3.4).  $\Fr[y]$ has the important property of
being a commutative unique factorization domain.  Every element $f\in\Aq$ has
a unique \emph{minimal central left composition (mclc)} \todo{Mark: Is the
  mclm equal to the norm?\\ MWG: Don't really know!}  $\fhat\in\cAq$, the
nonzero monic polynomial in $\cAq$ of minimal degree such that $\fhat=g \circ
f$ for some $g\in\Aq$.
Given $\nu\in\Frbar$, we say that $\nu$
\emph{belongs to} $f\in\Aq$ if $f$ is the nonzero polynomial in $\Aq$
of lowest degree of which $\nu$ is a root.

\begin{fact}[\citealp{Gie98}]
Let $p$ be a prime, $r$ a power of $p$ and $q=r^d$.
For $f\in\Aq$ of degree $r^n$, we can find the minimal central
left composition $\fhat\in\cAq$ with $O(n^3m^3)$ operations in $\Fr$.
\end{fact}

The following key theorem shows the close relationship between the
minimal central left composition and the minimal polynomial of the
Frobenius automorphism.
\begin{theorem}
\label{thm:minpolymclm}
Let $f\in\Aq$ be squarefree of degree $r^n$ with roots
$V_f\subseteq\Frbar$.  Fix an $\Fr$-basis
$\calB=\langle\nu_1,\ldots,\nu_n\rangle\in\Frbarn$ for $V_f$, so that
$V_f\iso\Fr^{n\times 1}$.  Let $S\in\Fr^{n\times n}$ represent the
action of the Frobenius automorphism $\sigma_{q}$ on $V_f$ with
respect to $\calB$.  Then the image $\tau(\fhat)\in\Fr[y]$ of the
minimal central left composition $\fhat\in\cAq$ of $f$ is equal to the
minimal polynomial $\Lambda\in F_r[x]$ of the matrix~$S$.
\end{theorem}


\begin{proof}

  First, suppose $\Lambda=\sum_{0\leq i\leq m}\Lambda_ix^i\in\F_r[x]$ is the minimal
  polynomial of $S$.  Then for all
  $\uvec=(u_1,\ldots,u_n)^t\in\F_r^{n\times 1}$,
  $0=\Lambda(S)=\Lambda(S)\uvec=\sum_{0\leq i\leq m}\Lambda_iS^i\uvec$.  Equivalently,
  if $L=\tau^{-1}(\Lambda)=\sum_{0\leq i\leq m}\Lambda_ix^{q^i}\in\cAq$ and 
  $u=\sum_{1\leq i\leq  n} u_i\nu_i\in V_f$ then
  $L(u)=\sum_{0\leq i\leq m} \Lambda_i\sigma_{q}^i(u)=0$, 
  and this holds for all $u\in V_f$.  
  Thus $L$ is a (central) left composition of $f$, and hence
  $\tau(\fhat)\divs\Lambda$, since $\fhat$ has minimal degree (and
  $\Fr[x]$ is a principal ideal domain).

  Conversely, suppose $\ghat=\sum_{0\leq i\leq
    d}\ghat_ix^{q^i}\in\cAq$ is \emph{any} central composition of $f$.
  So for all $w=\sum_{1\leq i\leq n}w_i\nu_i\in V_f$, $\ghat(w)=0$,
  and $\sum_{0\leq i\leq d}\ghat_i S^i\wvec=0$, where
  $\wvec=(w_1,\ldots,w_n)^t\in\Fr^{n\times 1}$, or equivalently
  $\tau(\ghat)(S)=0$.  Thus $\Lambda$ divides $\tau(\ghat)$, and hence
  $\Lambda\divs\tau(\fhat)$.
\end{proof}

We now present our algorithm to count collisions of polynomials in $\Aq$ of degree $r^2$.

\begin{marks-algorithm}{CollisionCounting} 

\Inspec $f\in\Aq$ of degree $r^2$, where $q = r^{d}$
\Outspec The number of collisions in decompositions of $f$

\Stmt[(1)] If $f'(0)=0$ Then
\Stmt[(2)] \> If $f=x^{r^2}$ Then Return 1
\Stmt[(3)] \> Else Return 2
\Stmt[] Else
\Stmt[(4)] \> $\fhat \gets \mclc(f)\in\cAq$
\Stmt[(5)] \> If $\deg\fhat=r$ Then Return $r+1$
\Stmt[(6)] \> Factor $\tau(\fhat)\in\Fr[y]$ over $\Fr[y]$
\Stmt[(7)] \> If $\tau(\fhat)\in\Fr[y]$ is irreducible Then Return 0
\Stmt[(8)] \> If $\tau(\fhat)=(y-a)^2$ for some $a\in\Fr$ Then Return 1
\Stmt[(9)] \> Return 2
\end{marks-algorithm}

The proof of the following is straightforward, using either the
factoring methods in $\Fr[y]$ from
\cite{canzas81} (probabilistic) or \cite{ron92} (deterministic,
assuming the ERH).

\begin{theorem}
The algorithm \texttt{\upshape CollisionCounting} works as specified and requires an expected
number of 
 $O(d^3)\log r$ operations in $\Fr$ using a randomized
algorithm, or 
$d^{O(1)}\log r$ operations with a deterministic
algorithm (assuming the ERH).
\end{theorem}


We note that the algorithm \texttt{CollisionCounting} also allows us to count the number of
rational roots of the projective polynomial $x^{r+1}+ax+b$.  This is
equal to the number of collisions of $x^{r^2}+ax^r+bx$, by \autoref{pro:equivalent}.



For the remainder of this section we look at the problem of counting the number of
irreducible right components of degree $r$ of any additive
polynomial $f\in\Aq$ of degree $r^n$.  The algorithm will run in time
polynomial in $n$ and $\log q$. This will also yield a fast algorithm
to compute the number of rational roots of a projective polynomial
$\Psi_n^{(a,b)}\in\Fq[x]$.

The approach is to compute explicitly the Jordan form of the Frobenius
operator $\sigma_{q}$ acting on the roots of $f$, as in
\eqref{eq:jordanform}.  We show how to do this quickly, despite the
fact that the actual roots of $f$ may lie in an extension of
exponential degree over $\Fq$.

\begin{marks-algorithm}{FindJordan}

\Inspec $f\in\Aq$ monic squarefree of degree $r^n$, where $r$ is a prime power
\Outspec Rational Jordan form $S\in\Fr^{n\times n}$ of the Frobenius automorphism
$\sigma_{q}(a)=a^q$ (for $a\in\Frbar$) on $V_f$, as in \eqref{eq:jordanform}

\Stmt[(1)] Compute $\fhat\gets\mclc(f)\in\cAq$
\Stmt[(2)] Factor $\tau(\fhat)\gets u_1^{\omega_1}u_2^{\omega_2}\cdots
u_\ell^{\omega_\ell}\in\Fr[y]$, where the $u_i\in\Fr[y]$ are monic irreducible
and pairwise distinct, and $\deg
u_i=s_i$  for $1\leq i\leq \ell$
\Stmt[(3)] For $i$ from 1 to $\ell$ do
\Stmt[(4)] \> For $j$ from 1 to $\omega_i$ do
\Stmt[(5)] \> \> $h_{ij} \gets \gcrc(\tau^{-1}(u_i^j),f)$ 
\Stmt[(6)] \> \> $\xi_{ij}\gets (\log_r h_{ij})/s_i$  (i.e., $\deg h_{ij}=r^{s_i\xi_{ij}}$)
\Stmt[(7)] \> For $j$ from 1 to $\omega_i-1$ do
\Stmt[(8)] \> \> $\delta_{ij} \gets \xi_{ij}-\xi_{i,j+1}$
\Stmt[(9)] \> $\delta_{i\omega_i} \gets \xi_{i\omega_i}$
\Stmt[(10)] \> $k_i \gets \xi_{i1}$
\Stmt[(11)] \> $(e_{i1},\ldots,e_{ik_i}) \gets (
\underbrace{1,\ldots,1}_{\delta_{i1}},
\underbrace{2,\ldots,2}_{\delta_{i2}}, \ldots,
\underbrace{\omega_i,\ldots,\omega_i}_{\delta_{i\omega_i}} )$
\Stmt[(12)] Return
$S = \diag\left(J_{\alpha_1}^{e_{11}},\ldots,J_{\alpha_1}^{e_{1k_1}},\ldots,
           J_{\alpha_\ell}^{e_{\ell
               1}},\ldots,J_{\alpha_\ell}^{e_{\ell k_\ell}}\right)$
\end{marks-algorithm}

\begin{theorem}
The algorithm \texttt{\upshape FindJordan} works as specified.  It requires
an expected  number of operations in $\Fq$ which is polynomial $n$ and $\log r$ (Las Vegas).
\end{theorem}
\begin{proof}
Note that the notation in the
algorithm corresponds directly to that of the rational Jordan form 
\eqref{eq:jordanform}.  In
Step 1, we know from Theorem \ref{thm:minpolymclm} that $\fhat$ is
the minimal polynomial of $S$.  Therefore all rational Jordan blocks
correspond to factors of $\fhat$ (determined in Step 2) and we only
need to figure out their multiplicities. 

For a particular $i$, we know by \cite{Gie98}, Theorem 4.4, that all
indecomposable components of $h_{ij}$ in $\Aq$ have degree
$s_i$.  Thus $\deg h_{ij}=r^{s_i\xi_{ij}}$ for an integer
$\xi_{ij}$.  As $i$ goes from $1$ to $\omega_i$, we determine the
number of eigenvalues with multiplicity $1$ or more ($\xi_{i1}$),
$2$ or more ($\xi_{i2}$), etc.  In Step 8, $\delta_{ij}$ is then the
number of Jordan blocks of $\alpha_i$ of multiplicity exactly $j$.
Doing this for all eigenvalues and all possible multiplicities
yields the final form in Step 10.

That the algorithm runs in polynomial time follows directly from the
fact that gcrc requires polynomial time (see \citealp{Gie98}), and
the factoring in  Step (2) requires polynomial time, say by \cite{canzas81}.
\end{proof}

Now given an $f\in\Aq$ we can quickly compute the rational Jordan form
of the Frobenius autormorphism on its root space.  Computing the
number of degree $r$ factors (or indeed, the number of irreducible
factors of any degree) is easy, following the same method as in
\autoref{sec:proj-polyn}.

\begin{theorem}
If the Frobenius automorphism of the root space of an $f\in\Aq$ has
rational Jordan form  in the notation of Algorithm
\texttt{\upshape FindJordan} where
\[
S =
\diag\left(J_{\alpha_1}^{e_{11}},\ldots,J_{\alpha_1}^{e_{1k_1}},\ldots,
 J_{\alpha_\ell}^{e_{\ell 1}},\ldots,J_{\alpha_\ell}^{e_{\ell
     k_\ell}}\right),
\]
\[
(e_{i1},\ldots,e_{ik_i}) \gets (
\underbrace{1,\ldots,1}_{\delta_{i1}},
\underbrace{2,\ldots,2}_{\delta_{i2}}, \ldots,
\underbrace{\omega_i,\ldots,\omega_i}_{\delta_{i\omega_i}} )
\]
for $1\leq i\leq\ell$,
then the number of indecomposable right components of degree $r$ is
\[
\sum_{i: s_i=1} \sum_{1\leq j\leq \omega_i} \delta_{ij}\cdot \frac{r^j-1}{r-1}.
\]
\end{theorem}

Thus,  the number of right components of degree $r$ of an
additive polynomial of degree $r^n$ can be computed in time polynomial in
$n$ and $\log q$.  Following \autoref{lem:projadd} we can also
determine the number of roots in $\Fr$ of a projective polynomial
$\Psi_n^{(a,b)}\in\Fr[x]$ in time polynomial in $n$ and $\log q$.

\section{Projective polynomials and roots}
\label{sec:proj-polyn-roots}








We now look to actually construct and enumerate all the polynomials in
each case 0, 1, 2, $r+1$ as in Theorem \ref{thm:countdecomp}.  For this, it is
useful to recall a little more about the ring $\Aq$.  The
following facts are from \cite{Ore33b}.
\begin{fact}
Let $f,g\in\Aq$. 
\begin{enumerate}
\item There exists a unique monic $h\in\Aq$ of maximal degree, and
$u,v,\in\Aq$, such that $f=u\circ h$ and $g=v\circ h$, 
called the \emph{greatest common right component (gcrc)}  of $f$ and
$g$.  Also, $h=\gcrc(f,g)=\gcd(f,h)$, and the roots of $h$ are
those in the intersection of the roots of $g$ and $h$.
\item There exists a unique monic and nonzero $h\in\Aq$ of minimal
degree, and $u,v\in\Aq$, such that $h=u\circ f$ and $h=v\circ g$,
called the \emph{least common left composition (lclc)} of
$f$, $g$.  The roots of $h$ are the $\Fr$-vector space sum of the
roots of $f$ and $g$; this sum is direct if $\gcrc(f,g)=1$.
\end{enumerate}
\end{fact}
In fact, there is an efficient Euclidean-like algorithm for computing
the lclc and gcrc; see, \cite{Ore33b}, and \cite{Gie98} for an analysis.

The main theorem counting the number of decompositions can now be shown. It is
equivalent to counting the number of times each case in
\autoref{thm:countdecomp} occurs.
\begin{theorem}
\label{thm:colcounts}
Let $r$ be a prime power and $q$ a power of $r$.  For $i\in\NN$ let
\begin{gather}
C_{q,r,m,i}=\{(a,b)\in\Fq^{2}: x^{r^2}+ax^r+bx \quad \text{has a maximal
  $i$-collision in $\Aq$} \}, \\
c_{q,r,m,i}=\# C_{q,r,m,i},
\end{gather}
and drop $q,r,m$ from the notation.  The following holds:
\begin{description}
\item[Case 0:]
\label{lem:colcount0}
$C_{0}$ is the set of all $f\in\Aq$ of degree $r^2$
whose minimal central left compositions $\fhat\in\cAq$ have degree $q^{2}$
and cannot be written as $\fhat=\ghat\circ\hhat$ for
$\ghat,\hhat\in\cAq$ of degree ${q}$, or equivalently that the image
$\tau(\fhat)\in\Fr[y]$ of $\fhat$ is irreducible of degree $2$.  We have
\[
c_0 =\frac{r(q^{2}-1)}{2(r+1)}.
\]
\item[Case 1:]
\label{lem:colcount1}
$C_{1}$ is the set of all $f \in \Aq$ of degree $r^2$ with minimal
central left composition $\fhat=\ghat\circ\ghat$ for $\ghat=x^{q}-cx$ for
$c\in\Fr^\times$, and
\[
c_1 = \frac{q^{2}-q}{r}+1.
\]
\item[Case 2:]
\label{lem:colcount2} $C_{2}$ is the set of all  $f \in \Aq$ with minimal central left composition
$\fhat=\ghat\circ\hhat$ for $\ghat,\hhat\in\cAq$ of degree ${q}$ with
$\gcd(\ghat,\hhat)=1$, and
\[
c_2 = \frac{(q-1)^2\cdot (r-2)}{2(r-1)}+q-1.
\]
\item[Case $\mathbold{r\mskip2mu +}$1:]
\label{lem:colcountr+1}
$C_{r+1}$ is the set of all  $f  \in \Aq$ of degree $r^2$ with minimal central
left composition $\fhat=x^{q}+cx$, for $c\in\Fr^\times$, and
\[
c_{r+1} = \frac{(q-1)(q-r)}{r(r^2-1)}.
\]
\end{description}
Since $c_0+c_1+c_2+c_{r+1}=q^{2}$, these are the only possible
numbers of collisions of a degree $r^2$ polynomial in $\Aq$.
\end{theorem}


\begin{proof}\quad
  \begin{description}
  \item[Case 0:] 
The number of irreducible polynomials in $\Fr[y]$ of degree $2$ is
$(r^2-r)/2$ (see \cite{LidNie83}).  Each polynomial
$\fhat\in\cAq$ of degree $r^{2m}$ has $r^{2m}-1$ nonzero roots, and hence
has $(r^{2m}-1)/(r^2-1)$ components in $\Aq$ of degree~$r^2$.

\item[Case 1:] 
Each such $f$ arises as a right component of degree $r^2$ of an
$\fhat=\ghat\circ\ghat\in\cAq$, for $\ghat=x^{q}+cx\in\cAq$, which is not a
right component of $\fhat$.  The number of roots
of $\ghat\circ\ghat$ which are not roots of $\ghat$ is $q^{2}-q$.
Each of these roots belongs to a polynomial in $f\in\Aq$ of degree
$r^2$, and each such $f$ has $r^2-r$ such roots which belong to that
$f$  (the other roots belong to a right component of degree $r$).  Thus
there are $(q^{2}-q)/(r^2-r)$ polynomials in $\Aq$ of degree $r^2$
whose minimal central left composition is $\fhat$.  There
are $r-1$ polynomials $\fhat$ of this form so there are $(q^{2}-q)/r$
polynomials $f\in\Aq$ with a unique decomposition.

\item[Case 2:]
We consider the case of polynomials with 2-collisions, and thus
whose minimal central left compositions have the form
$\fhat=\ghat\circ\hhat$, for $\ghat,\hhat\in\cAq$, with
$\gcd(\ghat,\hhat)=1$.

Each such $f\in\Aq$ has minimal central left composition
$\fhat=\ghat\circ\hhat\in\cAq$, for $\ghat,\hhat\in\cAq$ of degree
${q}$, with $\gcd(\ghat,\hhat)=1$. Thus we can construct an $f$ with
the desired properties by choosing a root $\nu$ of $\ghat$ and a
root $\omega$ of $\hhat$ and finding the $f\in\Aq$ which has both
$\nu$ and $\omega$ as roots (this corresponds to finding the
$g,h\in\Aq$ to which $\nu,\omega$ belong respectively, and letting
$f=\lclc(g,h)$).  Each of $\ghat,\hhat$ has $({q}-1)/(r-1)$ right
components of degree $r$, so for each choice of $\ghat,\hhat$ we have
$({q}-1)^2/(r-1)^2$ polynomials $f\in\Aq$ with the desired properties.
There are $\binom{r-1}{2}=(r-1)(r-2)/2$ distinct pairs of
$\ghat,\hhat$ with nonzero constant coefficient.

\item[Case $\mathbold{r\mskip2mu +}$1:] In this case the minimal central left composition of $f$ is
$\fhat=x^{q}-cx$ for some $c\in\Fr^\times$.  Thus, $\tau(\fhat)=y-c\in\F_r[y]$
is the minimal polynomial of the Frobenius automorphism $\sigma_{q}$ on
$V_\fhat$, the $\Fr$-vector space of $\fhat$, and all subspaces of
$V_\fhat$ are invariant under $\sigma_{q}$.  Hence each subspace is
exactly the set of roots of a polynomial in $\Aq$.  The number right
components $h\in\Aq$ of $\fhat$ of degree $r^2$ is the number of
$2$-dimension subspaces of $V_\fhat$.  The number of linearly
independent pairs of vectors in $V_\fhat$ is $({q}-1)({q}-r)$.  This is
the number of all bases for all vector spaces of dimension $2$.
Each $2$-dimensional vector space has $(r^2-1)(r^2-r)$ bases.  Thus
$\fhat$ has
\[
\frac{({q}-1)({q}-r)}{r(r-1)^2(r+1)}
\]
right components of degree $r^2$.  There are $(r-1)$
polynomials $\fhat$ of the form $x^{q}-cx$ for $c\in\Fq^\times$.
\end{description}
\end{proof}

We note that the proof is constructive and shows how to (efficiently) generate
polynomials in $\Aq$ of degree $r^2$ with a prescribed number of collisions.
In each case, the number of collisions of an $f\in\Aq$ is determined by the
factorization of its minimal central left composition $\fhat$ in $\cAq$. Here
$\deg\tau(\fhat)\in\{1,2\}$, and we can enumerate all such $\fhat$ in each
class (irreducible linear, irreducible quadratic, perfect square, or product
of distinct linear factors).  We can decompose each such $\fhat$ using the
algorithms of \cite{Gie98} to generate polynomials with a prescribed number of
collisions.


We show now how to construct indecomposable additive
polynomials of prescribed degree, and count their number.  We also
show how to construct additive polynomials with a single, unique
complete decomposition and count the number of such polynomials.

The following theorem characterizes indecomposable polynomials of
degree $r^\ell$ in terms of their minimal central left compositions. 
This theorem allows us to get hold of degree $r$ right components from the
roots of $\tau(\fhat)$ in $\Fq$.

\begin{theorem}[{\citealp[Theorem 4.3]{Gie98}}]
\label{thm:indectoirr}
Let $\fhat\in\cAq$ have degree $q^{\ell}$, such that
$\tau(\fhat)\in\Fr[y]$ is irreducible (of degree $\ell$).  Then
every indecomposable right component $f\in\Aq$ of $\fhat$
has degree $r^\ell$.  Conversely, all $f\in\Aq$ which are
indecomposable of degree $r^\ell$ are such that
$\tau(\fhat)\in\Fr[y]$ is irreducible of degree $\ell$, where
$\fhat\in\cAq$ is the minimal central left composition of~$f$.
\end{theorem}

The following bound has been shown in \cite{odo99}.  Our
methods here provide a simple proof.   Let
\[
I_r(n)=\sum_{d\divs n} \mu(n/d) r^d
\]
be the number of monic irreducible polynomials in $\Fr[y]$ of degree
$n$ (see, e.g., \citet{LidNie83}, Theorem 3.25).

\begin{theorem}
Let $q$ be a power of $r$.  The number of monic indecomposable polynomials $f\in\Aq$ of degree $r^n$ is
\[
\frac{q^{n}-1}{r^n-1}I_r(n).
\]
\end{theorem}
\begin{proof}
By Theorem \ref{thm:indectoirr} all such polynomials are right components of
polynomials $\fhat\in\cAq$ of degree $q^{n}$, where
$\tau(\fhat)\in\Fr[y]$ is irreducible (of degree $n$).  Any such
$\fhat$ has $(q^{n}-1)/(r^n-1)$ indecomposable right components in $\Aq$, all
of degree $r^n$.  There are $I_r(n)$ irreducible polynomials of
degree $n$ in $\Fr[y]$.
\end{proof}

Note that this implies there are (slightly) more indecomposable
additive polynomials of degree $r^n$ in $\Aq$ than irreducible
polynomials of degree $n$ in $\Fq[y]$.

The above theorem also yields a reduction from the problem of finding
indecomposable polynomials in $\Aq$ of prescribed degree to that of
decomposing polynomials in $\Aq$. A fast randomized algorithm for
decomposing additive polynomials is shown in \cite{Gie98}, which
requires a number of operations bounded above by $(n+m+\log
r)^{O(1)}$.  Thus, we can just choose a random polynomial in $\Aq$ of
prescribed degree and check if it is irreducible, with a high
expectation of success.  A somewhat slower polynomial-time reduction
from decomposing additive polynomials in $\Aq$ to factoring in
$\Fr[y]$ is also given in \cite{Gie98}.  This suggests the interesting
question as to whether one can find indecomposable polynomials in
$\Aq$ of prescribed degree $n$ in deterministic polynomial-time,
assuming the ERH (\`a la \citet{AdlLen86}).


We finish this section by establishing connections to the counts of \cite{blu04a} and \cite{gat09d}.

We have a prime $p$, integers $d$, $e$, and $m$ with $d$ dividing $e$,
$r=p^d$, $q=p^e$, set $\varphi_{r,m}= (r^m-1)/(r-1)$ and for $a,b \in \FF_q$
and $0 \leq i \leq \varphi_{r,m}$
\begin{equation*}
  \Psi_{m}^{(a,b)} = x^{\varphi_{r,m}}+ax+b.
\end{equation*}
This yields an equivalent description of $C_{q,r,m,i}$ by
\autoref{pro:equivalent} as
\begin{equation}
   \label{eq:3}
C_{q,r,m,i} =\{(a,b) \in \FF_q^2 \colon \text{$\Psi_{m}^{(a,b)}$ has exactly $i$
  roots in $\FF_q$}\}.  
\end{equation}
\autoref{sec:proj-polyn} says that
\begin{equation*}
  C_{q,r,m,i} \neq \emptyset \implies i \in S_{q,r,m}
\end{equation*}
and $S_{q,r,m}$ is determined in \autoref{thm:S}.  Furthermore, let
\begin{align*}
  C_{q,r,m,i}^{(1)} & = \{(a,b) \in C_{q,r,m,i} \colon b \neq 0\}, \\
  C_{q,r,m,i}^{(2)} & = \{(a,b) \in C_{q,r,m,i} \colon ab \neq 0\},
\end{align*}
and $c_{q,r,m,i}^{(j)} = \# C_{q,r,m,i}^{(j)}$ for $j =1,2$.  Leaving out the
indices, we have $C^{(2)} \subseteq C^{(1)} \subseteq C$.  The set $C^{(1)}$
occurs naturally in general decompositions (\autoref{pro:arbitr-coll}
\ref{item:iii} for $r=p$), and $C^{(2)}$ is the subject of \cite{blu04a}.  For
an integer $m \geq 1$, let
\begin{equation*}
  \gamma_{q,r,m} =   \gcd (\varphi_{r,m}, q-1).
\end{equation*}

\begin{proposition} \label{pro:C}
  We fix $q,r,m$ as above and drop them from the notation of $C_{q,r,m,i}$ and
  $c_{q,r,m,i}$.
  \begin{enumerate}
  \item\label{item:Ci} We have $C_i = C_i^{(1)}$ for all $i \notin \{1,
    \gamma_{m-1}+1\}$, and
    \begin{align*}
      C_1 \setminus C_1^{(1)} & = \{(a,0) \colon (-a)^{(q-1)/\gamma_{q,r,m-1}}
      \neq 1\}, \\
      C_{\gamma_{m-1}+1} \setminus C_{\gamma_{m-1}+1}^{(1)} & =  \{(a,0)
      \colon (-a)^{(q-1)/\gamma_{q,r,m-1}} = 1\} \\
      c_1 & =c_1^{(1)} + (q-1)(1-\gamma_{q,r,m-1}^{-1})+1, \\
      c_{\gamma_{m-1}+1}& = c_{\gamma_{m-1}+1}^{(1)} + (q-1)\gamma_{q,r,m-1}^{-1}.      
    \end{align*}
  \item\label{item:Cii} We have $C_i^{(1)} = C_i^{(2)}$ for all $i \notin \{0,
    \gamma_{m}\}$, and 
    \begin{align*}
      C_0^{(1)} \setminus C_0^{(2)} & = \{(0,b) \colon
      (-b)^{(q-1)/\gamma_{q,r,m}}\neq 1\}, \\
      C_{\gamma_m}^{(1)} \setminus C_{\gamma_m}^{(2)} & = \{(0,b) \colon
      (-b)^{(q-1)/\gamma_{q,r,m}} = 1\}, \\
      c_0^{(1)} & = c_0^{(2)} + (q-1)(1-\gamma_{q,r,m}^{-1}) \\
      c_{\gamma_m}^{(1)} & = c_{\gamma_m}^{(2)} + (q-1)\gamma_{q,r,m}^{-1}.
    \end{align*}
  \end{enumerate}
\end{proposition}

\begin{proof}
  \begin{enumerate}
  \item Let $i\in S_{q,r,m}$ and $(a,0) \in C_i \setminus C_i^{(1)}$ be
    arbitrary.  Then $\Psi_{m}^{(a,0)} = x^{\varphi_{r,m}}+ax =
    x(x^{r\varphi_{r,m-1}}+a)$.  Now $0$ is a root, and for $a=0$ it is the
    only one.  This places $(0,0)$ into $C_1 \setminus C_1^{(1)}$, and we may
    now assume $a \neq 0$.  Now let $t_0$ be a nonzero root of
    $\Psi_{m}^{(a,0)}$ and $t=t_0^r$.  Then $t^{\varphi_{r,m-1}} = -a$.

    Dropping the indices, we have $\varphi = \gamma \cdot (\varphi/\gamma)$
    from (\ref{eq:3}).  The power map $\pi_\gamma \colon w \mapsto w^\gamma$
    on $\FF_q^\times$ maps $\gamma$ elements to one, since $\gamma \mid (q-1)$.  Thus
    $\im \pi_\gamma$ is a group of order $(q-1)/\gamma$, and $\gcd
    (\varphi/\gamma, (q-1)/\gamma) = 1$.  Thus the $(\varphi/\gamma)$th power
    acts bijectively on this group, and $\im \pi_\gamma=\im \pi_\varphi$.  If there is one $t$ with $t^\varphi =
    -a$, then there are exactly $\gamma$ many.  Furthermore, we have
    \begin{equation*}
      -a \in \im \pi_\varphi  = \im \pi_\gamma
      \Longleftrightarrow (-a)^{(q-1)/\gamma} =1.
    \end{equation*}
    Together with the fact that the $r$th power acts bijectively on $\FF_q$,
    this shows that if $\Psi_{m}^{(a,0)}$ has at least one nonzero root, then
    it has exactly $\gamma$ roots.  Adding in the root $0$ shows the claims in
    \ref{item:Ci}.
\item Let $(0,b)\in \FF_q^2$ with $b\neq 0$ be an arbitrary element of
  $C^{(1)}\setminus C^{(2)}$.  Then $\Psi_{m}^{(0,b)} = x^{\varphi_{r,m}}+b$.  Now
  $0$ is not a root, but otherwise the argument for \ref{item:Ci} applies
  mutatis mutandis.
  \end{enumerate}
\end{proof}

%

We note that Theorem \ref{thm:colcounts} is also counting the number
of possible solutions to the equations $y^{r+1}+ay+b$, as in Bluher's
(2004) work.  For $m=2$, \eqref{eq:14} is equivalent to $h_{0}^{r+1}+ah_{0}+b=0$, so we are
counting the number of $h_{0}\in\Fq, q = r^{d}$ satisfying $y^{r+1}+ay+b=0$.  The
comparison with Bluher's work is interesting because she does not
consider the case $a=0$ or $b=0$ and because her work has multiple
cases depending on whether $d$ is even or odd and whether $m$ is even
or odd, whereas our counts have no such special cases.  

The result in the (relatively straightforward) case $a=0$ is consistent with
the more general Lemma 5.9 of \cite{gat08c}, where $q$ is not required to be a
power of $r$, but merely of $p$.


We now state as a corollary a result equivalent to that of
\cite{blu04a} (at least over $\Fq$, when $q=r^d$).

\begin{corollary}
\label{cor:bluher}
Let $r$ be a prime power, $d$ a positive integer and $q=r^{d}$.  Then 
\[
C_{q,r,2,i}^{(2)}=\{ (a,b)\in {\Fq^\times}^{2} \colon x^{r^2}+ax^r+bx\quad \text{has an $i$-collision} \},
\]
$C_{q,r,2,i}^{(2)}=\emptyset$ for $i\notin\{0,1,2,r+1\}$, and the following holds:
\begin{itemize}
\item[(i)] If $d$ is even, then $[c_0^{(2)},c_1^{(2)},c_2^{(2)},c_{r+1}^{(2)}]=$
\[
\left[
\frac{r ({q}-1)^2}{2(r+1)},
\frac{{q}({q}-1)}{r},
\frac{({q}-1)^2(r-2)}{2(r-1)},
\frac{({q}-1)({q}-r^2)}{r(r^2-1)}
\right].
\]
\item[(ii)] If $r$ is odd and $d$ is odd, then $[c_0^{(2)},c_1^{(2)},c_2^{(2)},c_{r+1}^{(2)}]=$
\begin{multline}
\left[
\frac{({qr}-1)({q}-1)}{2(r+1)},
\frac{{q}({q}-1)}{r},  \frac{({q}-1)({qr}-2{q}-2r+3)}{2(r-1)},
\frac{({q}-r)({q}-1)}{r(r^2-1)}
\right].
\end{multline}
\item[(iii)] If $r$ is even and $d$ is odd, then
$[c_0^{(2)},c_1^{(2)},c_2^{(2)},c_{r+1}^{(2)}]=$
\begin{multline}
\left[
\frac{r(q^{2}-1)}{2(r+1)},
\frac{({q}-1)({q}-r)}{r}, \frac{({q}-1)^2(r-2)}{2(r-1)},
\frac{({q}-r)({q}-1)}{r(r^2-1)}
\right].
\end{multline}
\end{itemize}
\end{corollary}
We note that each of these counts is $q-1$ times the corresponding count of
\citet[Theorem 5.6]{blu04a}, which projects down to a single parameter family.
We also note that the constructive nature of our proofs allows us to build
polynomials prescribed to be in any of these decomposition classes.  This
follows in the same manner as in the degree $r^2$ case (see the discussion
following Theorem 5.2).  We generate elements of $\cAq$ with the desired
factorization pattern (which determines the number of collisions) and
decompose these over $\Aq$ using the algorithms of \cite{Gie98}.

\section{General compositions of degree $r^{2}$}
\label{sec:general}

The previous sections provide a good
understanding of composition
collisions for additive polynomials.  We now move on to general
polynomials.   This section provides some explicit non-additive collisions.  

\begin{example}
We consider $\FF_{27} = \FF_3[y]/(m)$, with $m = y^3-y+1$, take $r=p=3$,
$u=1$, and let
\begin{equation*}
T = \{-1, -y^2, -y^2 -y-1, -y^2+y-1\}
\end{equation*}
consist of the $r+1$ roots of $t^{r+1}-ut+u$.  We obtain for
\begin{equation*}
f = x^9 + x^6 - x^5 + x^3 + x^2 + x
\end{equation*}
the following $4$-collision of monic original polynomials:
\begin{align*}
f & = (x^3 -x^2 + x ) \circ (x^3 - x^2 + x)\\
& = (x^3 + (y^2 + y -1)x^2 - (y +1)x) 
\circ (
x^3 - y^2x^2 + (y^2 - y)x) \\
& = (x^3 + (y^2 - y -1)x^2 -yx) 
 \circ (
x^3 - (y^2 + y + 1)x^2 + (y^2 -1)x) \\
& = (x^3 + (y^2 + 1)x^2 + (-y + 1)x)
\circ (
x^3 - (y^2 - y +1)x^2 + (y^2 + y)x).
\end{align*}
\end{example}

For any $f = \sum f_{i} x^{i} \in \Fq[x]$, we call $\deg_{2} f =
\deg(f- \lc(f)x^{\deg f})$ the \emph{second-degree} of $f$, with
$\deg_{2} f = -\infty$ for monomials and zero. Furthermore, $ f = g +
O(x^{k})$ with a polynomial $g \in \Fq [x]$ and an integer $k$, if
$\deg (f-g) \leq k$.

\begin{theorem}
\label{thm:nonadd}
Let $q$ and $r$ be powers of $p$, $\varepsilon \in \{0,1\}$, $u,
s\in \FF_{q}^{\times}$, $t \in T = \{t \in \FF_{q} \colon t^{r+1} -\varepsilon
ut + u = 0\}$, $\ell$ a positive divisor of $r-1$, $m = (r-1)/\ell$, and
\begin{align*}
f = F(\varepsilon, u, \ell, s) & = x (x^{\ell(r+1)}-\varepsilon u s^{r}x^{\ell} + us^{r+1})^{m}, \\
g = G(u, \ell, s, t) & = x (x^{\ell}-us^{r}t^{-1})^{m}, \\
h = H(\ell, s, t) & = x (x^{\ell}-st)^{m},
\end{align*}
all in $\FF_{q}[x]$.  Then
\begin{equation*}
f = g \circ h,
\end{equation*}
and $f$ is a $\# T$-collision.
\end{theorem}



\begin{proof}
From $u\neq 0$ follows $t \neq 0$, so 
that $g$ is well-defined. We find
  \begin{align*}
    g \circ h & = x (x^{\ell} -st)^{m}(x^{\ell}(x^{\ell}-st)^{r-1}-us^{r}t^{-1})^{m} \\
& = x ((x^{\ell}-st)^{r}x^{\ell}-(x^{\ell}-st)us^{r}t^{-1})^{m} \\
& = x (x^{\ell r + \ell} - s^{r} t^{r} x^{\ell} - u s^{r} t^{-1}x^{\ell} + u s^{r+1})^{m} \\
& = x(x^{\ell(r+1)} -s^{r}(t^{r}+ut^{-1})x^{\ell}+us^{r+1})^{m} \\
& = x (x^{\ell(r+1)}-\varepsilon u s^{r}x^{\ell} + us^{r+1})^{m} =f.
  \end{align*}
Note that $f$ is
independent of $t$.  We have different coefficients
\begin{align*}
g_{r-\ell} & = -mus^{r}t^{-1} \neq 0, \\
h_{r-\ell} & = -mst \neq 0,
\end{align*}
for different values of $t$, and therefore $\# T$ pairwise distinct decompositions of $f$.
\end{proof}

The polynomials described are additive if $\ell = r-1$.  If $\ell < r-1$, $r-\ell$ is not a
power of $r$ and $g_{r-\ell} \neq 0$, so that $g$ and $f$ are not additive.

If a polynomial $f \in \FF_q[x]$ is monic original, then so is $f_{(w)} = (x-f(w)) \circ f \circ (x+w)$ for
all $w\in
\FF_q$.  Every decomposition of $f$ induces a decomposition of
$f_{(w)}$ as specified below, and all $f_{(w)}$ have the same number of decompositions as $f_{(0)} = f$.

\begin{corollary}
\label{cor:transforms}
  We use the notation of \autoref{thm:nonadd}, an additional parameter
$w \in \FF_q$ and set
\begin{equation*}
\begin{split}
f_{(w)} = F(\varepsilon, u, \ell, s)_{(w)} & = (x-f(w)) \circ F(\varepsilon, u, \ell, s) \circ (x+w), \\
g_{(w)} = G(u, \ell, s, t)_{(w)} & = (x-f(w)) \circ G( u, \ell, s, t) \circ (x+h(w)), \\
h_{(w)} = H(\ell, s, t)_{(w)} & = (x- h(w)) \circ H(\ell, s, t) \circ (x+w).
\end{split}
\end{equation*}
Then $f_{(w)}=g_{(w)} \circ h_{(w)}$,
all three polynomials are monic original,
and \linebreak[4]$\{(g_{(w)}, h_{(w)}) \colon t \in
T\}$ is a $\# T$-collision.  
\end{corollary}

Among all $F(\varepsilon, u, \ell, s)_{(w)}$, the $F(\varepsilon, u, \ell,
s)_{(0)}$ is characterized by the vanishing of the coefficient of $x^{r^2-\ell r-\ell -1}$.

\begin{proposition}
\label{lem:unique1}
  Let $q$ and $r$ be powers of $p$.  Let $\varepsilon$, $u$, $\ell$, $s$, $t$ 
  and $\varepsilon^{*}$, $u^{*}$, $\ell^{*}$, $s^{*}$, $t^{*}$ satisfy the conditions
  of \autoref{thm:nonadd}, $w, w^{*} \in \Fq$, $f = F(\varepsilon, u, \ell,
  s)_{(w)}$, and $ f^{*} =  F
  (\varepsilon^{*}, u^{*}, \ell^{*}, s^{*})_{(w^{*})}$. The following holds:
\begin{enumerate}
\item\label{item:i} If $f=f^{*}$, then $\varepsilon=\varepsilon^{*}$ and $\ell=\ell^{*}$.
\item\label{item:ii} If $\varepsilon=0$ and $\ell = r- 1$,then $f= F(0, -1, r-1,
       st)_{(0)}$ and $f=f^{*}$ if and only if $(s/s^{*})^{r+1} = 1$.
\item\label{item:iii} If $\varepsilon = 0$ and $\ell < r-1$, then $f =
       F(0,-1, \ell, st)_{(w)}$ and $f=f^{*}$ if and only if $w=w^{*}$ and
       $(s/s^{*})^{r+1} = 1$.
\item\label{item:iv} If $\varepsilon = 1$ and $\ell = r-1$, then $f = F(1, u, r-1,
      s)_{(0)}$ and $f=f^{*}$ if and only if $u = u^{*}$ and $s=s^{*}$.
\item\label{item:v} If $\varepsilon = 1$ and $\ell < r-1$, then $f=f^{*}$ if and only if
      $u = u^{*}$, $s=s^{*}$ and $w = w^{*}$.
\end{enumerate}
\end{proposition}


\begin{proof}
We have
\begin{align*}
  f & = F(\varepsilon, u, \ell, s)_{(w)} \\
& = x(x^{\ell(r+1)m} - m\varepsilon us^{r}x^{\ell(r+1)(m-1)+\ell} +
  mu s^{r+1}x^{\ell (r+1)(m-1)} \\
& \quad + O(x^{\ell (r+1)(m-2)+2\ell})) \\
&= x^{r^{2}} - m \varepsilon u s^{r} x^{r^{2}-\ell r} +m u s^{r+1}
x^{r^{2}-\ell r - \ell} + O(x^{r^{2}-2\ell r}),
\end{align*}
\begin{equation}
\label{eq:36}
\begin{split}
  f_{r^{2}-\ell r} & = -m\varepsilon us^{r},  \\
  f_{r^{2} - \ell r-\ell} & = mus^{r+1}.
\end{split}
\end{equation}
Therefore
\begin{equation*}
\deg_{2} f = \begin{cases}
r^{2} - \ell r & \text{ if } \varepsilon =1, \\
r^{2} -\ell r - \ell & \text{ if } \varepsilon = 0.
\end{cases}
\end{equation*}
Furthermore, $p \nmid r-1 = \ell m$, so that $p \nmid \ell$. We have $\deg_{2} f = \deg_{2} f^{(*)}$ and
  $\varepsilon = 1$ if and only if $r$ divides $\deg_{2} f$. For both values
  of $\varepsilon$, $\deg_{2} f$ determines $\ell$ uniquely.  This proves \ref{item:i}.

For $\ell = r-1$, $f$ is additive and therefore
  \begin{align*}
    f & =  F(\varepsilon, u, r-1, s)_{(w)}\\
& = (x - F(\varepsilon, u, r-1,s)
   (w))  \circ F(\varepsilon, u, r-1,  s)(x) \circ (x+w) \\
 & = (x - F(\varepsilon, u, r-1,  s)(w))   \circ ( F(\varepsilon, u, r-1, 
 s)(x) + F(\varepsilon, u, r-1,  s)(w)) \\
&  =F(\varepsilon, u, r-1,  s)_{(0)}
  \end{align*}
for all $w \in \Fq$.

For $\ell < r-1$ the coefficient of $x^{r^{2}-\ell r - \ell - 1}$ in  $F(\varepsilon, u,
  \ell, s)_{(w)}$ equals
  \begin{equation*}
F(\varepsilon, u, \ell, s)_{r^{2}-\ell r - \ell - 1} +
    w (r^{2}-\ell r -\ell)mus^{r+1},
  \end{equation*}
and $(r^{2}-\ell r -\ell)mus^{r+1} \neq 0$.  Therefore, $F(\varepsilon, u, \ell, s)_{(w)}  = F(\varepsilon,
  u, \ell, s)_{(w^{*})}$ if and only if $w = w^{*}$.

 For $\varepsilon=1$, we find from \eqref{eq:36} that  $s =  -
  f_{r^{2}-\ell r - \ell} /   f_{r^{2}-\ell r} $ and $u =   f_{r^{2} -\ell r} / (
  -ms^{r})$ depend only on $f$.

 For $\varepsilon = 0$, we have $t^{r+1} = -u$ and 
  \begin{equation*}
    F(0, u,\ell, s)_{(w)} = \left(x (x^{\ell(r+1)} - (st)^{r+1})^{m}\right)_{(w)} =  F(0, -1, \ell, st)_{(w)}.
  \end{equation*}

Consider $F(0, -1,
      \ell, s)_{(w)} = F(0, -1,
      \ell, s^{*})_{(w)}$, divide by $x$, extract $m$th roots and find by
      coefficient comparison $s^{r+1} = {s^{*}}^{r+1}$.

Combining the observations for $\ell = r -1$, $\ell < r-1$ and $\varepsilon =0
$, $\varepsilon = 1$, respectively proves the claims for the four cases \ref{item:ii}-\ref{item:v}.


\end{proof}

\begin{corollary}
  Let $p, q, r$ as in \autoref{thm:nonadd}, $\gamma = \gcd(r+1, q-1)$, $i \in \{2, r+1\}$, and $N_{i}$ the
  number of $i$-collisions of the form described in \autoref{cor:transforms}.
  Then 
  \begin{equation*}
    N_{i}  =  \left(1 -q  + q \cdot d(r-1) \right) \left( c_{q,r,i}^{(2)} + \delta_{\gamma, i} \frac{
        q-1}{\gamma}  \right),
  \end{equation*}
  where $d(r-1)$ is the number of divisors of $r-1$, $\delta_{i,j}$ is Kronecker's delta,  and $c_{q,r,i}^{(2)}$
  are determined in \autoref{cor:bluher}.
\end{corollary}

\begin{proof}
  For $\varepsilon =0$, $f$ is an $i$-collision, only if $y^{r+1}=1$ has
  exactly $i$ solutions, according to \autoref{lem:unique1} \ref{item:ii} and
  \ref{item:iii}.  Generally, this equation has exactly $\gamma = \gcd (r+1,
  q-1)$ solutions in $\Fq^{\times}$.  Furthermore there are $(q-1)/\gamma$
  values for $s \in \Fq^{\times}$ which yield pairwise different
  $s^{r+1}$. The number of $i$-collisions of the form described in
  \ref{item:ii} is therefore $\delta_{\gamma, i} \cdot (q-1) / \gamma$, and of
  the form described in \ref{item:iii}
$ \delta_{\gamma, i}   q (d(r-1)-1)(q-1) / \gamma$, tacking into account the
$(d(r-1)-1)$ possible divisors $\ell$ and $q$ choices for $w$.

For $\varepsilon = 1$, we have to consider $u$, such that $y^{r+1} - uy + u
\in \Fq [y]$ has exactly $i$ roots.  Let $a,b \in
  \Fq^{\times}$ and $u = a^{r+1}b^{-r}$.  The invertible transformation $x
  \mapsto y= -ab^{-1}x$ gives a
  bijection
  \begin{equation*}
    \{ t \in \Fq^{\times} \colon  t^{r+1} - ut + u = 0\} \leftrightarrow \{ \tau \in
    \Fq^{\times} \colon \tau^{r+1} + a\tau + b = 0\}.
  \end{equation*}
Every value of $u$ corresponds to exactly $q-1$ pairs $(a,b)$, namely an
arbitrary $a \in \Fq^{\times}$ and $b$ uniquely determined as $b^{r} =
u^{-1}a^{r+1}$.  \autoref{pro:equivalent} and the definition of $c_{q,r,i}^{(2)}$
yield $c_{q,r,i}^{(2)}/(q-1)$
 values for $u$. Therefore the number of $i$-collisions is $c_{q,r,i}^{(2)}$ for the form described
 in \ref{item:iv}, and
$c_{q,r,i}^{(2)}q (d(r-1)-1)$
for the form described in \ref{item:v}.
\end{proof}

\Citet{gat08c}, Lemma~3.29, determines $\gcd(r+1, q-1) $ explicitely.

\begin{conjecture}
\label{conj:NonaddAsConstructed}
Any maximal $i$-collision with $i \geq 2$ at degree $p^{2}$ is
either a Frobenius collision or of the form described in \autoref{cor:transforms}.
\end{conjecture}
The conjecture has been experimentally verified for $q \leq 9$ using Sage.

There are $q^{r-1}$ Frobenius collisions and all but $x^{r^{2}}=x^{r} \circ
x^{r}$ are maximal 2-collisions.  The number of maximal $i$-collisions with
$i\geq 2$ is therefore bounded from below by
\begin{equation*}
 N_{2} + N_{r+1} + q^{r-1} -1. 
\end{equation*}
The conjecture claims that this is also an upper bound.




In the following, we present partial results on this conjecture, concentrating
on the simplest case $r=p$.  We also give an upper bound on the number of
decompositions a single polynomial can have in the case of degree $p^{2}$.  No
nontrivial estimate seems to be in the literature.


\begin{proposition}
\label{pro:arbitr-coll}
Let $C$ be a non-Frobenius $i$-collision over $\FF_{q}$ with $i \geq 2$ at degree
$p^{2}$.  There is an integer $k$ with $1 \leq k < p$ and the
following properties for all $(g,h) \in C$.
\begin{enumerate}
\item\label{item:i} $\deg_{2} (g) = \deg_{2} (h) = k$.
\item\label{item:ii} For all $(g^{*},h^{*}) \in C$ with $(g,h) \neq (g^{*},
h^{*})$, we have $g_{k} \neq g^{*}_{k}$ and $h_{k} \neq h^{*}_{k}$.
\item\label{item:iii} Set $a=-f_{kp}$ and $b = k^{-1}f_{kp-p+k}$. Then
$b h_k \neq 0$, and
\begin{gather}
h_{k}^{p+1}+ah_{k}+b = 0 \\
g_{k}=-a-h_{k}^{p} = bh_{k}^{-1}.    
\end{gather}
\item\label{item:iv} $i \leq p+1$.
\end{enumerate}
\end{proposition}


\begin{proof}
  We write
\begin{align*}
g & = x^p + g_\ell x^\ell + \dots + g_1 x, \\
h & = x^p + h_m x^{m}  + \dots + h_1 x, \\
f = g\circ h & = x^{p^{2}} + f_{{p^{2}-1}}x^{p^{2}-1} + \dots + f_{1}x,
\end{align*}
with all $f_{i}, g_{i}, h_{i}\in \FF_{q}, 1 \leq \ell,m < p$ and $g_\ell h_m
\neq 0$.  For $u,v \in \FF_{q}[x]$ and $e \in \NN$, we write $u = v +
O(x^{e})$ if $\deg (u-v) \leq e$.  Similarly, $(O(x^{e}))^{p}$ indicates
a polynomial $w$ with $\deg w \leq 
e$ such that $u=v+w^{p}$.

The highest terms in $h^{\ell}$ and $g \circ h$ are
\begin{align}
h^{\ell} & = (x^{p}+h_{m}x^{m}+O(x^{m-1}))^{\ell}  \nonumber \\
& = x^{\ell p} +
\ell h_{m}x^{(\ell-1)p+m} + O(x^{(\ell-1)p+m-1}), \nonumber \\
g \circ h & =  x^{p^{2}} + h_{m}^{p} x^{m p} +
(O(x^{m-1}))^{p} + g_{\ell} x^{\ell p} + \ell g_{\ell}h_{m}x^{(\ell-1)p+m} \nonumber \\
& \quad +
O(x^{(\ell-1)p+m -1}) + O(x^{(\ell-1)p}). \label{eq:star}
\end{align}
Thus the highest term $f_{i}x^{i}$ in $f$ with $f_{i} \neq 0$ and $p
\nmid i$ occurs for $i = (\ell-1)p + m$.  Since $1 \leq \ell,m <p$
$(\ell,m)$ is determined by $f$ and identical for all $(g,h) \in C$.
Algorithm 4.9 of \cite{gat09b} computes
the components $g$ and $h$ from $f$, provided that $h_{p-1}\neq 0$. We
do not assume this, but can apply the same method.  Once $g_{\ell}$ and
$h_{m}$ are determined, the remaining coefficients first of $h$,
then of $g$, are computed by solving a linear equation of the form
$uh_{i}=v$, where $u$ and $v$ are known at that point, and $u \neq
0$.  Quite generally, $g$ is determined by $f$ and $h$.  Now take some
$(g^{*},h^{*})\in C$. If $(g_{\ell},h_{m})=(g^{*}_{\ell}, h^{*}_{m})$,
then $(g,h)=(g^{*},h^{*})$ by the uniqueness of the procedure just sketched.  
Inspection of the coefficient of
$x^{(\ell-1)p+m}$ in \eqref{eq:star} shows that $g_{\ell} = g^{*}_{\ell}$ if
and only if $h_{m}= h^{*}_{m}$.  Furthermore, $\deg_{2}(g\circ
h)$ is either $m p$ or $\ell p$.  If these two integers are
distinct, then either $h_{m}^{p}$ (and hence $h_{m}$) is
determined by $f$, namely if $m > \ell$, and otherwise $g_{\ell}$ is.  In
either case, we can conclude from the above that $(g,h) =
(g^{*},h^{*})$.  Since $(g,h) \neq (g^*, h^*)$ this shows $\ell= m$, and
\ref{item:i} and \ref{item:ii} for $k=\ell$.

For \ref{item:iii}, we find from \eqref{eq:star},
\begin{align*}
f_{kp} & = h_{k}^{p} + g_{k} , \\
f_{kp-p+k} & = kg_{k}h_{k} = kh_{k} (f_{kp}-h_{k}^{p}) = -kh_{k}^{p+1} +kf_{kp}h_{k}.
\end{align*}

The $i$ distinct (see \ref{item:ii}) values $h_{k}^{(i)}$ are solutions to a
degree $p+1$ equation in $h_{k}$.  This proves \ref{item:iv}.

\end{proof}

We have $k=1$ for additive polynomials, and $k=r - \ell$ in
\autoref{thm:nonadd}.

\begin{proposition}
\label{pro:nosmallk}
Let $C$ be a non-Frobenius $i$-collision over $\Fq$ with $i \geq 2$ at degree
$p^{2}$, and $k$ the integer defined in Proposition~\ref{pro:arbitr-coll}. Then $k = 1$ or $k > p/2$.
\end{proposition}


\begin{proof}
  We expand $h^{k}$ some further
\begin{align*}
h^{k} & = (x^{p} + h_{k} x^{k} + h_{k-1} x^{k-1} + \dots + h_{1}
x)^{k} \\
& = x^{kp} + kx^{p(k-1)}(h_{k}x^{k}+\dots + h_{1}x) \\
& \quad + \binom{k}{2}x^{p(k-2)}(h_{k}x^{k}+\dots + h_{1}x)^{2} +
O(x^{p(k-3)+3k}).
\end{align*}
The coefficient of $x^{kp-2p+2k}$ is $\binom{k}{2}h_{k}^{2}$ from the
last line, plus $kx^{kp-p+i}\cdot h_{i}$ if $kp-p+i=kp-2p+2k$ from
the previous line.  The latter means $i=2k-p$.  Now assume that $k \leq
p/2$. Then $i \leq 0$, so that only the last line contributes.  No
other summand in $g\circ h$ contributes to the coefficient of
$x^{kp-2p+2k}$ in $f$, and therefore
\begin{align*}
f_{kp-p+k} & = kg_{k}h_{k} , \\
f_{kp-2p+2k} & = g_{k} \binom{k}{2} h_{k}^{2} = \binom{k}{2} k^{-1}
f_{kp-p+k} h_{k}.
\end{align*}
The binomial coefficient and $f_{kp-p+k}$ are nonzero, and it follows that $h_{k}$ has
the same value for all $(g,h)\in C$.  By Proposition~\ref{pro:arbitr-coll}\ref{item:ii}, this is
false. 
\end{proof}

This shows that there are no collisions at degree $p^2$ with $k=2$ if $p>3$ nor with
$k=3$ if $p>5$.

\section{Conclusion and open questions}\label{sec:conclusion}
We have presented composition collisions with component degrees $(r,r)$  for
polynomials $f$ of degree $r^{2}$,
and observed a fascinating interplay between these examples---quite
distinct in the additive and the $f_{r^2-r-1}\neq 0$ cases---and
\citeauthor{abh97}'s
projective polynomials and \citeauthor{blu04a}'s statistics on their
roots. Furthermore, we showed that our examples comprise all
possibilities in the additive case, and provided large classes of examples in
general.  Showing the completeness of our examples in the general case is the
main challenge left open here as \ref{conj:NonaddAsConstructed}. 

Generalizations go in two directions. One is degree $r^{k}$ for $k
\geq 3$. Additive polynomials are of special interest here, and the
rational normal form of the Frobenius automorphism will play a major
role. For general polynomials, the approximate counting problem is solved in \cite{gat09b} with
a relative error of about $q^{-1}$, and it is desirable to reduce this, say to
$q^{-r+1}$.

The second direction is to look at degree $ar^{2}$ with $r \nmid
a$. Now there are no additive polynomials, but for approximate
counting, the best known relative error can be as large as $1$. It
would be interesting to also push this below $q^{-1}$, or even
$q^{-r+1}$.

In some sections, we assume the field size $q$ to be a power of the parameter
$r$.  As in \citeauthor{blu04a}'s (2004) work, our methods go
through for the general situation, where $q$ and $r$ are independent powers of
the characteristic.  

With respect to additive polynomials, a more thorough computational
investigation of projective polynomials is warranted.  Automatic generation of
Bluher-like equations for higher degree projective polynomials should be
possible, as would be a more exact understanding of their possible collision
numbers.

\section{Acknowledgments}

The authors thank Toni Bluher for telling us about the  applications of projective
polynomials, and an anonymous referee for pointing us to \cite{HelKho10}.

The work of Joachim von zur Gathen and Konstantin Ziegler was
supported by the B-IT Foundation and the Land Nordrhein-Westfalen. The
work of Mark Giesbrecht was supported by NSERC Canada and MITACS.

\bibliographystyle{cmc}

\bibliography{journals,refs,lncs}

 \providecommand{\ymd}[3]{{\day=#3\relax\month=#2\relax\year=#1\relax
  \number\day~\ifcase\month\or January\or February\or March\or April\or May\or
  June\or July\or August\or September\or October\or November\or
  December\fi\space\number\year}} \providecommand{\hide}[1]{.}
  \providecommand{\Hide}[1]{\unskip} \providecommand{\gobble}[1]{}
  \providecommand{\todo}[1]{\textbf{`?`?`?#1???}} \providecommand{\Name}[1]{#1}
  \providecommand{\Textgreek}[1]{\textgreek{#1}}
  \providecommand{\at}{\char64\relax}
  \providecommand{\cyr}{\PackageError{cyrillic}{Package not loaded. Use
  \string\usepackage{cyrillic} to define \string\cyr\space appropriately}{}
  \gdef\cyr{\def\cprime{c'}{\bf ?cyr?}}\cyr}
  \makeatletter\protected@write\@auxout{}{\string
  \gdef\string\abbr{\string\csname\space
  @gobble\string\endcsname}}\gdef\abbr{}\makeatother
\begin{thebibliography}{33}
\providecommand{\natexlab}[1]{#1}
\providecommand{\url}[1]{\texttt{#1}}
\providecommand{\urlprefix}{URL }
\expandafter\ifx\csname urlstyle\endcsname\relax
  \providecommand{\doi}[1]{doi:\discretionary{}{}{}#1}\else
  \providecommand{\doi}{doi:\discretionary{}{}{}\begingroup
  \urlstyle{rm}\Url}\fi

\bibitem[{Abhyankar(1997)}]{abh97}
\textsc{Shreeram~S. Abhyankar}.
\newblock Projective Polynomials.
\newblock \emph{Proceedings of the American Mathematical Society},
  \textbf{125}(6):1643--1650, 1997.
\newblock ISSN 00029939.
\newblock \urlprefix\url{http://www.jstor.org/stable/2162203}.

\bibitem[{Adleman \& Lenstra(1986)}]{AdlLen86}
\textsc{Leonard~M. Adleman} \& \textsc{Hendrik~W. Lenstra, Jr.}
\newblock Finding Irreducible Polynomials over Finite Fields.
\newblock In \emph{Proceedings of the Eighteenth Annual ACM Symposium on Theory
  of Computing, {\rm Berkeley~CA}}, pages 350--355. ACM Press, 1986.

\bibitem[{Barton \& Zippel(1985)}]{barzip85}
\textsc{David~R. Barton} \& \textsc{Richard Zippel}.
\newblock Polynomial Decomposition Algorithms.
\newblock \emph{Journal of Symbolic Computation}, \textbf{1}:159--168, 1985.

\bibitem[{Bluher(2003)}]{Blu03}
\textsc{Antonia~W. Bluher}.
\newblock On $x^{6}+x+a$ in Characteristic Three.
\newblock \emph{Designs, Codes and Cryptography}, \textbf{30}:85--95, 2003.
\newblock
  \urlprefix\url{http://www.springerlink.com/content/r213567443r63360/fulltext%
.pdf}.

\bibitem[{Bluher(2004{\natexlab{a}})}]{Blu04b}
\textsc{Antonia~W. Bluher}.
\newblock Explicit formulas for strong Davenport pairs.
\newblock \emph{Acta Arithmetica}, \textbf{112}(4):397--403,
  2004{\natexlab{a}}.

\bibitem[{Bluher(2004{\natexlab{b}})}]{blu04a}
\textsc{Antonia~W. Bluher}.
\newblock On $x^{q+1}+ax+b$.
\newblock \emph{Finite Fields and Their Applications}, \textbf{10}(3):285--305,
  2004{\natexlab{b}}.
\newblock \urlprefix\url{http://dx.doi.org/10.1016/j.ffa.2003.08.004}.

\bibitem[{Bluher \& Lasjaunias(2006)}]{Blulas06}
\textsc{Antonia~W. Bluher} \& \textsc{Alain Lasjaunias}.
\newblock Hyperquadratic power series of degree four.
\newblock \emph{Acta Arithmetica}, \textbf{124}(3):257--268, 2006.

\bibitem[{Cade(1985)}]{cad85}
\textsc{John~J. Cade}.
\newblock A New Public-key Cipher Which Allows Signatures.
\newblock In \emph{Proceedings of the 2nd SIAM Conference on Applied Linear
  Algebra}, page Raleigh NC A11. SIAM, 1985.

\bibitem[{Cantor \& Zassenhaus(1981)}]{canzas81}
\textsc{David~G. Cantor} \& \textsc{Hans Zassenhaus}.
\newblock A New Algorithm for Factoring Polynomials Over Finite Fields.
\newblock \emph{Mathematics of Computation}, \textbf{36}(154):587--592, 1981.

\bibitem[{Dillon(2002)}]{Dil02}
\textsc{J.~F. Dillon}.
\newblock Geometry, codes and difference sets: exceptional connections.
\newblock In \emph{Codes and designs ({C}olumbus, {OH}, 2000)}, volume~10 of
  \emph{Ohio State Univ. Math. Res. Inst. Publ.}, pages 73--85. de Gruyter,
  Berlin, 2002.
\newblock \doi{10.1515/9783110198119.73}.
\newblock \urlprefix\url{http://dx.doi.org/10.1515/9783110198119.73}.

\bibitem[{Dorey \& Whaples(1974)}]{dorwha74}
\textsc{F.~Dorey} \& \textsc{G.~Whaples}.
\newblock Prime and Composite Polynomials.
\newblock \emph{Journal of Algebra}, \textbf{28}:88--101, 1974.
\newblock \urlprefix\url{http://dx.doi.org/10.1016/0021-8693(74)90023-4}.

\bibitem[{von~zur Gathen(1990{\natexlab{a}})}]{gat90c}
\textsc{Joachim von~zur Gathen}.
\newblock Functional Decomposition of Polynomials: the Tame Case.
\newblock \emph{Journal of Symbolic Computation}, \textbf{9}:281--299,
  1990{\natexlab{a}}.
\newblock \urlprefix\url{http://dx.doi.org/10.1016/S0747-7171(08)80014-4}.

\bibitem[{von~zur Gathen(1990{\natexlab{b}})}]{gat90d}
\textsc{Joachim von~zur Gathen}.
\newblock Functional Decomposition of Polynomials: the Wild Case.
\newblock \emph{Journal of Symbolic Computation}, \textbf{10}:437--452,
  1990{\natexlab{b}}.
\newblock \urlprefix\url{http://dx.doi.org/10.1016/S0747-7171(08)80054-5}.

\bibitem[{von~zur Gathen(2008)}]{gat08c}
\textsc{Joachim von~zur Gathen}.
\newblock Counting decomposable univariate polynomials.
\newblock \emph{Preprint}, page 92 pages, 2008.
\newblock \urlprefix\url{http://arxiv.org/abs/0901.0054}.

\bibitem[{von~zur Gathen(2009{\natexlab{a}})}]{gat09d}
\textsc{Joachim von~zur Gathen}.
\newblock An algorithm for decomposing univariate wild polynomials.
\newblock \emph{Submitted}, page 32 pages, 2009{\natexlab{a}}.

\bibitem[{von~zur Gathen(2009{\natexlab{b}})}]{gat09b}
\textsc{Joachim von~zur Gathen}.
\newblock The Number of Decomposable Univariate Polynomials.
\newblock In \textsc{John~P. May}, editor, \emph{Proceedings of the 2009
  International Symposium on Symbolic and Algebraic Computation ISSAC2009, {\rm
  Seoul, Korea}}, pages 359--366. 2009{\natexlab{b}}.
\newblock ISBN 978-1-60558-609-0.

\bibitem[{Giesbrecht(1988)}]{gie88}
\textsc{Mark~William Giesbrecht}.
\newblock Complexity Results on the Functional Decomposition of Polynomials.
\newblock Technical Report 209/88, University of Toronto, Department of
  Computer Science, Toronto, Ontario, Canada, 1988.
\newblock Available as \url{http://arxiv.org/abs/1004.5433}.

\bibitem[{Giesbrecht(1995)}]{gie95}
\textsc{Mark~{\Hide{William}} Giesbrecht}.
\newblock Nearly Optimal Algorithms for Canonical Matrix Forms.
\newblock \emph{SIAM J. Comp.}, \textbf{24}:948--969, 1995.

\bibitem[{Giesbrecht(1998)}]{Gie98}
\textsc{Mark~{\Hide{William}} Giesbrecht}.
\newblock Factoring in Skew-Polynomial Rings over Finite Fields.
\newblock \emph{Journal of Symbolic Computation}, \textbf{26}(4):463--486,
  1998.
\newblock \urlprefix\url{http://dx.doi.org/10.1006/jsco.1998.0224}.

\bibitem[{Hardy \& Ramanujan(1918)}]{harram18}
\textsc{G.~H. Hardy} \& \textsc{S.~Ramanujan}.
\newblock Asymptotic formulae in combinatory analysis.
\newblock \emph{Proceedings of the London Mathematical Society},
  \textbf{17}(2):75--115, 1918.

\bibitem[{Helleseth \& Kholosha(2010)}]{HelKho10}
\textsc{Tor Helleseth} \& \textsc{Alexander Kholosha}.
\newblock x$^{\mbox{2$^{\mbox{l}}$+1}}$+x+a and related affine polynomials over
  GF (2$^{\mbox{{\it }}}$).
\newblock \emph{Cryptography and Communications}, \textbf{2}(1):85--109, 2010.

\bibitem[{Helleseth, Kholosha \& Johanssen(2008)}]{HelKho08}
\textsc{Tor Helleseth}, \textsc{Alexander Kholosha} \& \textsc{Aina Johanssen}.
\newblock \textit{m}-Sequences of Different Lengths with Four-Valued Cross
  Correlation.
\newblock \emph{IEEE International Symposium on Information Theory}, 2008.

\bibitem[{Kozen \& Landau(1986)}]{kozlan86}
\textsc{Dexter Kozen} \& \textsc{Susan Landau}.
\newblock Polynomial Decomposition Algorithms.
\newblock Technical Report 86-773, Department of Computer Science, Cornell
  University, Ithaca~NY, 1986.

\bibitem[{Kozen \& Landau(1989)}]{kozlan89}
\textsc{Dexter Kozen} \& \textsc{Susan Landau}.
\newblock Polynomial Decomposition Algorithms.
\newblock \emph{Journal of Symbolic Computation}, \textbf{7}:445--456, 1989.
\newblock An earlier version was published as \cite{kozlan86}.

\bibitem[{Landau \& Miller(1985)}]{lanmil85}
\textsc{S.~Landau} \& \textsc{G.~L. Miller}.
\newblock Solvability by Radicals is in Polynomial Time.
\newblock \emph{Journal of Computer and System Sciences}, \textbf{30}:179--208,
  1985.

\bibitem[{Lidl \& Niederreiter(1983)}]{LidNie83}
\textsc{Rudolf Lidl} \& \textsc{Harald Niederreiter}.
\newblock \emph{Finite Fields}.
\newblock Number~20 in Encyclopedia of Mathematics and its Applications.
  Addison-Wesley, Reading~MA, 1983.

\bibitem[{Odoni(1999)}]{odo99}
\textsc{Robert Winston~Keith Odoni}.
\newblock On additive polynomials over a finite field.
\newblock \emph{Proceedings of the Edinburgh Mathematical Society},
  \textbf{42}:1--16, 1999.

\bibitem[{Ore(1933)}]{Ore33b}
\textsc{O.~Ore}.
\newblock On a Special Class of Polynomials.
\newblock \emph{Transactions of the American Mathematical Society},
  \textbf{35}:559--584, 1933.

\bibitem[{R{\'{o}}nyai(1992)}]{ron92}
\textsc{L{\hide{ajos}}~R{\'{o}}nyai}.
\newblock Galois groups and Factoring Polynomials over Finite Fields.
\newblock \emph{SIAM Journal on Discrete Mathematics}, \textbf{5}:345--365,
  1992.

\bibitem[{Schinzel(1982)}]{sch82c}
\textsc{Andrzej Schinzel}.
\newblock \emph{Selected Topics on Polynomials}.
\newblock Ann Arbor; The University of Michigan Press, 1982.
\newblock ISBN 0-472-08026-1.

\bibitem[{Schinzel(2000)}]{sch00c}
\textsc{Andrzej Schinzel}.
\newblock \emph{Polynomials with special regard to reducibility}.
\newblock Cambridge University Press, Cambridge, UK, 2000.
\newblock ISBN 0521662257.

\bibitem[{Zeng, Li \& Hu(2008)}]{ZenLi08}
\textsc{Xiangyong Zeng}, \textsc{Nian Li} \& \textsc{Lei Hu}.
\newblock A class of nonbinary codes and their weight distribution.
\newblock \emph{ArXiv e-prints, arxiv 0802.3430v1}, 2008.
\newblock
  \urlprefix\url{http://arxiv.org/PS_cache/arxiv/pdf/0802/0802.3430v1.pdf}.

\bibitem[{Zippel(1991)}]{zip91}
\textsc{Richard Zippel}.
\newblock Rational Function Decomposition.
\newblock In \textsc{Stephen~M. Watt}, editor, \emph{Proceedings of the 1991
  International Symposium on Symbolic and Algebraic Computation ISSAC~'91, {\rm
  Bonn, Germany}}, pages 1--6. ACM Press, Bonn, Germany, 1991.
\newblock ISBN 0-89791-437-6.

\end{thebibliography}


\end{document}